\documentclass[a4paper]{amsart}
\usepackage[T1]{fontenc}
\usepackage{lmodern}
\usepackage{amssymb}
\usepackage[all]{xy}
\usepackage[inline,shortlabels]{enumitem}
\usepackage{nicefrac,stmaryrd}
\usepackage{graphicx, import}
\usepackage{mathtools}
\usepackage[british]{babel}
\usepackage{microtype}
\usepackage{verbatim}

\usepackage[pdftitle={Topological fundamental groupoid.I.},
pdfauthor={RD Holkar, Md A Hossain},
  pdfsubject={}
]{hyperref}
\usepackage[lite]{amsrefs}
 \usepackage{amsthm,amsfonts}
 \usepackage{nicefrac}
 \usepackage{caption}%

\numberwithin{equation}{section}
\theoremstyle{plain}
\newtheorem{theorem}[equation]{Theorem}
\newtheorem*{theorem*}{Theorem}
\newtheorem{lemma}[equation]{Lemma}
\newtheorem*{lemma*}{Lemma}
\newtheorem{proposition}[equation]{Proposition}
\newtheorem*{proposition*}{Proposition}
\newtheorem{propdef}[equation]{Proposition and Definition}
\newtheorem{corollary}[equation]{Corollary}
\newtheorem*{corollary*}{Corollary}

\theoremstyle{definition}
\newtheorem{definition}[equation]{Definition}
\newtheorem{observation}[equation]{Observation}

\theoremstyle{remark}
\newtheorem{remark}[equation]{Remark}
\newtheorem{example}[equation]{Example}
\newtheorem{property}[equation]{Property}

\newcommand*{\defeq}{\mathrel{\vcentcolon=}}
\newcommand{\N}{\mathbb N}
\newcommand{\R}{\mathbb R}
\newcommand{\C}{\mathbb C}
\newcommand*{\base}[1][G]{{#1}^{(0)}}

\newcommand{\nb}{\nobreakdash} 
\newcommand{\7}{\backslash}
\newcommand{\inverse}{^{-1}}
\newcommand{\homeo}{\approx}
\newcommand{\iso}{\simeq}

\newcommand{\FGd}{\mathrm{\Pi}_1}
\newcommand{\FGp}{\mathrm{\pi}_1}
\newcommand{\TCat}{\mathcal{T}}
\newcommand{\GCat}{\mathcal{G}^{\mathrm{T}}}
\newcommand{\GCatAlg}{\mathcal{G}^{\mathrm{A}}}

\newcommand{\Cont}{\mathrm{C}}
\newcommand{\Basic}{\mathcal{B}}
\newcommand{\UInt}{\mathbb{I}}
\newcommand{\Sph}{\mathbb{S}}
\newcommand{\PathS}[1]{\mathrm{P}\!#1}
\newcommand{\etale}{{\'e}tale}

\newcommand*{\Id}{\textup{id}}
\newcommand*{\Cst}{\textup C^*}

\title[Topological fundamental groupoid. I.]{Topological fundamental
  groupoid. I.}
\author{Rohit
  Dilip Holkar} \email{rohit.d.holkar@gmail.com} \address{Department
  of Mathematics, Indian Institute of Science Education and Research
  Bhopal, Bhopal Bypass Road, Bhauri, Bhopal 462 066, Madhya Pradesh,
  India.}  \date{\today}
\author{Md Amir Hossain} \email{mdamir18@iiserb.ac.in} \address{Department
  of Mathematics, Indian Institute of Science Education and Research
  Bhopal, Bhopal Bypass Road, Bhauri, Bhopal 462 066, Madhya Pradesh,
  India.}  \date{\today}

\thanks{\emph{Subject class.} 55P, 55Q, 14H30, 22E67, 22A22, 54A10, 57S}

\begin{document}
\maketitle{}

\begin{abstract}
  We show that the fundamental groupoid~\(\Pi_1(X)\) of a locally path
  connected semilocally simply connected space~\(X\) can be equipped
  with a \emph{natural} topology so that it becomes a topological
  groupoid; we also justify the necessity and minimality of these two
  hypotheses on~\(X\) in order to topologise the fundamental
  groupoid. We find that contrary to a belief---especially among the
  Operator Algebraists---the fundamental groupoid is not
  {\etale}. Further, we prove that the fundamental groupoid of a
  topological group, in particular a Lie group, is a
  \emph{transformation groupoid}; again, this result disproves a
  standard belief that the fundamental groupoids are \emph{far} away
  from being transformation groupoids. We also discuss the point-set
  topology on the fundamental groupoid with the intention of making it
  a locally compact groupoid.
\end{abstract}

\setcounter{tocdepth}{1}
\tableofcontents{}

\section*{Introduction}

This article has grown out of the curiosity of studying fundamental
groupoids as topological groupoids. The goal of the project is two
fold: firstly, to describe when a fundamental groupoid is locally
compact and can be equipped with a Haar system in order to study its
\(\Cst\)\nb-algebra. Secondly, to study actions of this topological
groupoid. The current article is the first goal partially, in the
sense that it describes the topology on the fundamental groupoid. The
Haar system on it is discussed
in~\cite{Holkar-Hossain2023Top-FGd-III}, and
in~\cite{Holkar-Hossain2023Top-FGd-II} we identify a certain
\emph{natural} action category of the fundamental groupoid~\(\FGd(X)\)
with that of covering spaces of~\(X\).

When it comes to topological groupoids, the \(\Cst\)\nb-algebraists
have extensively studied or used the following groupoids:
transformation groupoid of a group(oid) action; groupoid associated
with an inverse semigroup; groupoid associated with a graph; and
holonomy groupoid of a foliation. However, the fundamental groupoid
seems to have failed to gather sufficient attention of the Operator
Algebraist. The negligence of the fundamental groupoid by the Operator
Algebraist is justifiable as Muhly, Renault and Williams show in their
famous article~\cite{Muhly-Renault-Williams1987Gpd-equivalence} that
this groupoid is equivalent to the fundamental group of the
space. Therefore, for \(\Cst\)\nb-algebraic purposes, the study of the
\(\Cst\)\nb-algebra of a fundamental groupoid can be
replaced---\emph{upto} the Morita equivalence---by that of the
corresponding fundamental group.

It seems that a \emph{standard} set of hypotheses and basic
topological results for studying topological fundamental groupoids are
missing. For example, multiple articles discussing various types of
covering spaces or path lifting properties or path spaces can be
found\cite{Fischer-Zastrow2007Gen-Uni-Coverings-and-Shape-Gp},\cite{Bogley-Sieradski1998Uni-Path-Spaces},\cite{Brodskiy-etale2012Cov-Maps-for-loc-path-connected-spaces}.
Various topologies on the covering space are studied,
e.g.~\cite{Virk-Zastrow2014Comparison-of-Topologies-on-Cov-Space}. Although
a claim about the topology of the fundamental group was falsely proved
in~\cite{Biss2002Fundamental-Gpd-Top}, similar claims (which
fortunately turn out to be correct!)  about the topology of the simply
connected covering spaces also pop-up \emph{momentarily} in standard
textbooks of Munkres~\cite[Page 45, last
paragraph]{Munkress1975Topology-book} and Spanier~\cite[Proof of
Theorem 13, page 82, \S 5, Chapter 2]{Spanier1966Alg-Top-Book}
\emph{without} any justification.

Above various topologies on covering spaces or fundamental groups, and
unjustified claims suggest multiple choices for the topology on the
fundamental groupoid. Furthermore, unjustified claims and folklores
make it challenging to choose the appropriate candidate for this
topology. On top of it, we could not find any literature that offers a
standard set of hypotheses and results for a \emph{natural} topology
on the fundamental groupoids. In the current article, we choose a
standard set of hypotheses and establish a framework for a
\emph{natural} topology on fundamental groupoids.

\subsection*{The topology on the fundamental groupoid}
\label{sec:topol-fund-group}

In general, the fundamental groupoid of a path connected (and even
locally path connected) space may fail to be topological with a
natural choice of topology; this topology is the quotient topology
induced by the compact-open topology on the path space. The reason is
that if the fundamental groupoid were a topological groupoid with this
topology, then isotropy of a point---which is the fundamental group of
the space---must be a topological group. Fabel shows that, the
fundamental group of the Hawaiian Earring fails to be
topological~\cite{Fabel2011HE-Counterexample}.

Thus the question is: what are the \emph{optimal} assumptions on the
topology of a space so that its fundamental groupoid is topological?
Needless to say that there are certain natural choices for the
topology on the fundamental groupoid for which these assumptions are
sought after! In fact, there are more than one candidates for these
natural topologies!

Once the assumptions on the topology of the space are set, the next
task is to compare the different natural topologies on the fundamental
groupoid. Fortunately, all these topologies agree!

Next, we discuss the last two paragraphs in detail, and in the end
also give two topological reasons for studying the \emph{topological}
fundamental groupoids.

\medskip

\paragraph{\itshape The optimal hypotheses:}

It turns out that locally path connectedness and semilocally simply
connectedness of a space~\(X\) are the optimal hypotheses to
topologise the fundamental groupoid~\(\FGd(X)\). The optimality of
these assumptions is essentially justified by the existence of the
simply connected covering space of~\(X\). For the sake of clarity, we
follow Hatcher's definition of semilocally simply connected space,
see~\cite[second last paragraph on page 63]{Hatcher2002Alg-Top-Book}.

Recall the definition of the \emph{universal} covering space
(\cite{Spanier1966Alg-Top-Book}*{page 80}) which is the universal
object in the category of connected covering spaces (\cite[Chapter 2,
Sec.5]{Spanier1966Alg-Top-Book}). For example, a \emph{simply
  connected} covering space is the universal covering
space~\cite[Corollary 6.2,
Chapter~V]{Massey1991Basic-Course-in-Alg-Top}. Recall the standard
construction of a simply connected covering
space~\cite{Hatcher2002Alg-Top-Book}; in this construction, the
underlying set consists of the paths starting at a fixed point.

Now assume that~\(X\) is a path connected topological space whose
fundamental groupoid~\(\FGd(X)\) is topological; let \(s_{\FGd(X)}\)
be the source map of this groupoid. Fix \(x_0\in X\). Then we have a
contentment
\(\FGp(X,x_0) \subseteq s_{\FGd(X)}\inverse(x_0)\subseteq
\FGd(X)\). Thus, it is natural to expect that topology on the
fundamental groupoid should make the fundamental group \(\FGp(X,x_0)\)
a topological group and, we should expect that the subspace topology
should make \(s_{\FGd(X)}\inverse(x_0)\) the simply connected
cover. This expectation implies that we should assume \emph{at least}
those conditions on the topology of~\(X\) which assure existence of
the simply connected cover. What assures the existence of a simply
connected cover\footnote{It is noteworthy that
  Spanier~\cite{Spanier1966Alg-Top-Book} discusses the existence of
  the universal covering space that is not necessarily simply
  connected. The topic of various topologies on the universal covering
  space is still being discussed,
  e.g.~\cite{Fischer-Zastrow2007Gen-Uni-Coverings-and-Shape-Gp},~\cite{Virk-Zastrow2014Comparison-of-Topologies-on-Cov-Space},~\cite{Bogley-Sieradski1998Uni-Path-Spaces},~\cite{Brodskiy-etale2012Cov-Maps-for-loc-path-connected-spaces}. However,
  similar to the case of topologies on the fundamental group, we do
  not venture to discuss any of these `generalised' topologies on the
  fundamental groupoid which yield `generalised' covering spaces. %
}?

Hilton and Wylie~\cite[Example~6.6.14, pp
258--259]{Hilton-Wylie1960Intro-to-Alg-top} give an example of a space
that fails to have a covering space because it is not locally path
connected; in this example, the absence of locally path connectedness
causes the failure of the path lifting property. Thus our spaces must
be \emph{locally path connected}.

Next, Spanier~\cite[Corollary 14, Sect.5, Chapter
2]{Spanier1966Alg-Top-Book} shows that a simply connected covering
space of a path connected, locally path connected space exists
\emph{iff} the space is \emph{semilocally simply
  connected}\footnote{Instead of `semilocally simply connected',
  Spanier uses the term `semilocally 1-connected' which is defined on
  page~78 before Theorem 10 in~\cite[Sect.5, Chapter
  2]{Spanier1966Alg-Top-Book}. This is same as `semilocally simply
  connected' in our convention adopted from
  Hatcher~\cite{Hatcher2002Alg-Top-Book}.}. Thus for a path connected
space, locally path connectedness and semilocally simply connectedness
are necessary properties to enjoy a simply connected covering space,
hence a universal covering space.

In fact, it is well-known that locally path connectedness and
semilocally simply connectedness are foundational hypotheses for
assuring the existence of simply connected cover.

These two hypotheses play an important role in topologising the
\emph{fundamental group}: Theorem~2
in~\cite{Fabel2007Top-Fund-Gp-of-Metric-Spaces} and
Theorem~1.1~\cite{Calcut-McCarthy2009Top-Fund-Gp} highlights the
important of these hypotheses in order to topologise the fundamental
group---in a certain obvious way\footnote{Some recent articles
  topologises the fundamental group in different ways, for example
  see~\cite{Virk-Zastrow2014Comparison-of-Topologies-on-Cov-Space},~\cite{Bogley-Sieradski1998Uni-Path-Spaces}. These
  topologies induce different types of continuities of the group
  product. All these continuities are weaker than the usual
  simulteneous continuity of the product \(G\times G\to G\) in both
  variable. We do not enter these notions of \emph{nearly} topological
  groups; our goal is to make the fundamental group a topological
  group in the usual sense.}---of a locally path connected metrizable
spaces.

The above discussion should have made the naturality, necessity and
optimality of local path connectedness and semilocally simply
connectedness in order to topologise the fundamental groupoid. Hence,
we stick to such spaces.

\medskip
  
\paragraph{\itshape The three topologies on \(\FGd(X)\):}

After fixing the basic assumptions, now first discuss two
\emph{natural} or \emph{obvious} choices of topologies on the
fundamental groupoid.
  
The first choice of the topology is insinuated by the expectation that
the universal cover becomes a subspace of the fundamental
groupoid. This topology, which we call the UC topology, is defined by
mimicking the topology in the construction of the universal
cover. Here `UC' should remind us of the universal cover. To add more
about this topology, the remnant of UC topology can be seen in
Hurewicz's
article~\cite{Hurewicz1935Homotopie-Homologie-Und-Lokaler-Zusammenang}
and the later in Dugunji's
article~\cite{Dugunji1950Topogilized-Fund-Gp}.

The other topology, called CO', is the following: the fundamental
groupoid of a space~\(X\) is a quotient set \(\Cont(I,X)\), the set of
of all continuous functions from the unit interval~\(I\) to~\(X\) by
the path homotopy equivalence relation. The function
space~\(\Cont(I,X)\) comes with a natural topology, namely, the
compact-open topology. The CO' topology is the quotient topology
on~\(\FGd(X)\) induced by the compact-open topology.

We show that for our choice of spaces, these two topologies
coincide~Proposition~\ref{prop:CO-UC-tops-same}. We also show that
this topology makes the fundamental groupoid a topological groupoid
Theorem~\ref{thm:fund-gpd-top}.  Interestingly, not these topologies
but their restrictions to the covering spaces and the fundamental
group are very well-studied.  We elaborate on this in the latter
section titled `The many topologies'.

Let~\(X\) be path connected, locally path connected and semilocally
simply connected space. Using the famous equivalence
(\cite{Muhly-Renault-Williams1987Gpd-equivalence}) between the
fundament group and fundamental groupoid via the universal covering
space~\(\tilde{X}\), one can obtain the fundamental groupoid as a
quotient of the space~\(\tilde{X}\times \tilde{X}\) by the diagonal
action of the fundamental group,
see~Proposition~\ref{exa:fund-gp-action}. This is another way of
constructing the fundamental groupoid. We consider this the
\emph{third} topology on~\(\FGd(X)\). By the way, this construction is
not new; Reinhart describes it in~\cite[Proposition
2.37]{Reinhart1983Folliations-book}. This construction also follows
directly from a well-known fact about equivalent groupoids~\cite[Lemma
2.38, pp 58--59]{Williams2019A-Toolkit-Gpd-algebra}.

\medskip

\paragraph{\itshape Uses of the topology on the fundamental groupoid:}

Topologising the fundamental groupoid has a classical motivation: the
notion of a topological group can be traced back to 1935 when Hurewicz
introduced a notion of topologised fundamental group.  Dugunji
followed this idea in
1950~\cite{Dugunji1950Topogilized-Fund-Gp}. \emph{Topological
  fundamental groupoid} is certainly a strong successor of this
classical quest. Moreover, standardising a topological groupoid also
standardises the notion of topological fundamental group.

Apart from our original motivation, here are two instances which show
that the topological fundamental groupoid can be a better invariant:
in~\cite{Fabel2006Top-Fundamental-Gpd-dist-iso-homotopy-gps}, Fabel
shows that the topologised fundamental group is a finer invariant than
the algebraic one as it can distinguish spaces with isomorphic
homotopy groups. In Remark~\ref{rem:Top-FGd-stroger-than-Alg-FGp}, we
given an example of a map of spaces~\(f\) which induces an isomorphism
of algebraic fundamental groupoids but the map is not an isomorphism
of topological fundamental groupoids. Thus, topological fundamental
groupoid is a stronger invariant than the algebraic one.

\subsection*{Some other issues}

In this section, we describe the problematic or unclear issues we
faced about fundamental groupoid and our solutions to it.  \medskip

\paragraph{\itshape The many topologies:}

Recall the earlier mentioned UC and CO' topologies on the fundamental
groupoid in Proposition~\ref{prop:CO-UC-tops-same}.
Interestingly---in the context of path connected, locally path
connected and semilocally simply connected spaces, in particular, for
manifolds---this Proposition~\ref{prop:CO-UC-tops-same} \emph{seems} a
folklore, possibly, known to the experts. Our literature survey makes
us believe the last remark: when it comes to standard textbooks, all
authors Spanier~\cite{Spanier1966Alg-Top-Book},
Dieck~\cite{tom-Dieck2008Alg-Top-book},
Munkres~\cite{Munkress1975Topology-book},
Hatcher~\cite{Hatcher2002Alg-Top-Book},
Massey~\cite{Massey1991Basic-Course-in-Alg-Top}\cite{Massey1977Alg-Top-Intro-reprint},
Rotman~\cite{Rotman1988Alg-Top-Book} and
Fulton~\cite{Fulton1995Alg-Top-book} use the UC topology for
constructing the \emph{universal covering space}. Spanier~\cite[Proof
of Thorem 13, page 82, \S 5, Chapter 2]{Spanier1966Alg-Top-Book} and
Munkres~\cite[Page 45, last paragraph]{Munkress1975Topology-book}
mention the existence of~CO' topology on the universal cover but do
not offer more details. At the same time, in case of the
\emph{fundamental groupoid}, Brown~\cite[Lemma
2]{Brown-Danesh-1975-Top-FG-1} uses the UC topology on the groupoid.
Whereas, Reinhart or Muhly choose the CO' one, see~\cite[Proposition
2.37 and the discussion preceding it]{Reinhart1983Folliations-book}
or~\cite[Example~5.33.4]{Muhly1999Coordinates}.  A relatively recent
article of Fischer and Zastrow~\cite[Lemma
2.1]{Fischer-Zastrow2007Gen-Uni-Coverings-and-Shape-Gp} show the
equivalence of these two topologies for the universal cover.  However,
we could not find an explicit proof that these topologies on the
\emph{fundamental groupoid} are same. Rather, we suspect that the
curent article could be the first one to pen the proof. Our proof of
Proposition~\ref{prop:CO-UC-tops-same} is essentially based
on~\cite[Lemma
2.1]{Fischer-Zastrow2007Gen-Uni-Coverings-and-Shape-Gp}.

Although, comparison between the UC and CO' topologies on the
fundamental groupoid seems unavailable, their comparison on the
universal cover are well-studied,
see~\cite{Virk-Zastrow2014Comparison-of-Topologies-on-Cov-Space}.

In contrast to the comparison of the of the UC and CO' topologies on
the fundamental groupoid, their restrictions to the universal cover
and the fundamental group are well-studies well-known topics, for
example the article of Virk and
Zastrow~\cite{Virk-Zastrow2014Comparison-of-Topologies-on-Cov-Space}
discusses this in detail and in a general setting.

\medskip

\paragraph{\itshape A false folklore and corresponding confusion:}

The lack of enough attention and careful treatment of topological
fundamental groupoid has lead to a false folklore, which we have heard
among some Operator Algebraist. The folklore is that the fundamental
groupoid is {\etale}\footnote{In fact, the modest expectation that the
  simply connected covering space should be a subspace of the
  fundamental groupoid rules out the possibility of discreteness of
  the fibres over the range or source maps.}. This claim can also be
found in expert's work as well, for example~\cite[Example
2.1.4]{Khalkhali2013Basic-NCG-Book}.

But, interestingly, the first\footnote{As per our knowledge}
article~\cite[Example 2.3]{Muhly-Renault-Williams1987Gpd-equivalence},
written by Muhly, Renault and Williams, to mention topological
groupoid and the one which often Operator Algebraists refer to in
connection with topological fundamental groupoid does not fall for
this folklore.
  
One possible origin of this misconception could be the confusion
between a `groupo\"ide\ localement trivial' and an `{\etale}
groupoid'. Ehresmann defines the prior~\cite[page 143, second
paragraph]{Ehresmann1958Categories-Top-et-categories-diff} and the
latter one is well-known, for example see~\cite[first paragraph in \S
1]{Neshveyev2013KMS-states}. The phrase \emph{groupo\"ide\ localement
  trivial} is translated as `locally trivial groupoid'
see~\cite[second paragraph in main article on page
1]{Westman1967Locally-trivial-Cr-gpd}. The definitions of locally
trivial groupoid and \etale\ one imply that these are two distinct
classes. In fact, we notice that the fundamental topological groupoid
is locally trivial~(\cite{Brown-Danesh-1975-Top-FG-1}) but not \etale,
see~Remark~\ref{rem-Fgd-local-trivial-not-etale}.

Muhly, Renault and Williams~\cite[Example
2.3]{Muhly-Renault-Williams1987Gpd-equivalence} credit
Reinhart~\cite{Reinhart1983Folliations-book} for their example. And in
the literature of groupoids among Operator Algebraist, it is Muhly's
notes~\cite[Example~5.33.4]{Muhly1999Coordinates} where one finds an
explicit mention of a locally trivial groupoid in a way that the
reader can differentiate it from an \etale\ groupoid.

\medskip

\paragraph{\itshape A partial breach of a belief:}

The fundamental groupoid has been believed to be a `genuine'
groupoid---in the sense that it is not a transformation groupoid (or,
vaguely speaking, it is not obtained from a \emph{group}). Contrary to
this belief, we show that the fundamental groupoid of a topological
group, in particular a Lie group, is a transformation groupoid!
See~Theorem~\ref{thm:FGd-of-gp:as-tranf}.

\subsection*{Other results}
\label{sec:other-results}

Apart from establishing a standard topology on the fundamental
groupoid, we describe the point-set topology of it in
Section~\ref{sec:prop-topol-fg}. We essentially show that the
topological properties such as Hausdorffness, locally compactness and
second countability are borne by the fundamental groupoid \emph{iff}
either the underlying space of the universal covering space bears
them.

\subsection*{Organisation of the current article: }
\label{sec:current-article}

The article has three main sections.
Section~\ref{sec:notat-conv-prel} establishes notation, convention and
preliminaries for topological groupoids. We also discuss actions and
equivalences of groupoids. In Subsection~\ref{sec:comp-open-topol}, we
list or prove some useful properties of the compact-open topology.

Section~\ref{sec:fund-group-as} starts be describing and then
comparing the~UC and~CO' topologies;
Proposition~\ref{prop:CO-UC-tops-same} is the main result about this
comparison. Corollary~\ref{cor:quotient-from-PX-to-PiX-open} shows
that the range map for fundamental groupoid is open.

In Subsection~\ref{sec:topol-fund-group-1},
Theorem~\ref{thm:fund-gpd-top} proves that the fundamental groupoid is
topological; its Corollary~\ref{cor:subspace-top-on:fibre-UC-top}
shows that the topology on the fundamental groupoid is compatible with
the universal covering space and the fundamental group. Then next
Subsection~\ref{sec:functoriality} describes the functoriality of
topological fundamental groupoids. Finally,
Theorem~\ref{thm:FGd-of-gp:as-tranf} in
Section~\ref{sec:fund-group-group} proves that the fundamental
groupoid of a group is a transformation groupoid.

The last section, Section~\ref{sec:descr-fund-group}, we first give an
alternate description of the fundamental groupoid in
Proposition~\ref{prop:quotient-description-of-fund-gp}. Then in
Section~\ref{sec:prop-topol-fg} describes the point-set topology of
the fundamental groupoid in Propositions~\ref{prop:Haus-FGd},
\ref{prop:local-cpt}, \ref{prop:sec-cble} and~\ref{prop:paracpt}.

\medskip

\section{Notation, conventions and preliminaries:}
\label{sec:notat-conv-prel}

\subsection{Groupoids}
\label{sec:groupoids-prel}

We shall work with topological groupoids.

For locally compact spaces and groupoids we follow Tu's conventions
in~\cite{Tu2004NonHausdorff-gpd-proper-actions-and-K}. Thus by a
quasi-compact space we mean a space whose open covers have finite
subcovers. A quasi-compact and Hausdorff space is called
\emph{compact}. A topological space is called locally compact if every
point in it has a compact neighbourhood. Thus a locally compact space
is locally Hausdorff and hence~\(\mathrm{T}_1\).  Note that a compact
set~\(K\) is also locally compact Hausdorff. Then the open
set~\(K^0\subseteq K\), the interior of~\(K\), is a locally compact
Hausdorff space~\cite[Corollary 29.3]{Munkress1975Topology-book}. A
space is called \emph{paracompact} if it is Hausdorff and its every
open covering has a locally finite open refinement.
\begin{lemma}
  \label{lem:loc-cpt}
  Assume that~\(X\) is a locally compact space. Suppose \(x\in X\) and
  an open neighbourhood~\(U\) of~\(x\) are given. Then~\(U\) contains
  an open neighbourhood~\(V\) of~\(x\) such that~\(\overline{V}\) is
  compact and~\(\overline{V}\subseteq U\).
\end{lemma}
\begin{proof}
  Let \(K\) be a compact neighbourhood of~\(x\); then its
  interior~\(K^0\) is a \emph{nonempty} Hausdorff and locally compact
  open neighbourhood of~\(x\). Thus~\(W = K^0\cap U\) is an open
  subspace of the compact space~\(K\) containing~\(x\); hence~\(W\) is
  a locally compact Hausdorff open set in~\(X\). Therefore, \(W\)
  contains an open neighbourhood~\(V\) of~\(x\) such that
  \(\overline{V}\) is compact and contained in~\(W\). Since
  \(W\subseteq U\), \(V\) is the desired open neighbourhood of~\(x\).
\end{proof}

For us, a \emph{groupoid} is a \emph{small} category in which every
arrow is invertible. Thus a groupoid can be denoted as a
quintuple~\((G,\base[G], r, s, \mathrm{inv})\) where \(G\) is the
small category with~\(\base[G]\) as the set of objects;~\(r\)
and~\(s\) are the range and source maps
\(G \rightrightarrows \base[G]\); and \(\mathrm{inv}\colon G\to G\) is
the inversion map. Thus, for an arrow \(\gamma \in G\), \(r(\gamma)\)
and \(s(\gamma)\) and \(\mathrm{inv}(\gamma)\) are, respectively, the
range of~\(\gamma\), the source of~\(\gamma\)
and~\(\gamma\inverse\). However, we shall not write the quintuple all
the time but simply call~\(G\) as the groupoid. Note that
\(\mathrm{inv}\circ \mathrm{inv}\) is the identity mapping on~\(G\),
and \(r = s\circ \mathrm{inv}\).

For a groupoid~\(G\), we always consider~\(\base[G]\) as a subset
of~\(G\) by identifying it with the unit arrows on the objects. This
convention implies that \(s(\gamma) = \gamma\inverse\gamma\) and
\(r(\gamma)= \gamma\gamma\inverse\). Finally, we call the elements
of~\(\base[G]\) the units instead of objects.

For above groupoid~\(G\), an element of the fibre product
\(G\times_{s, \base[G], r} G =\{(\gamma,\eta)\in G\times G : s(\gamma)
= r(\gamma) \}\) is called a \emph{composable pair} of arrows
in~\(G\). Moreover, the multiplication on~\(G\) is the mapping
\[
  G\times_{s, \base[G], r} G \to G, \quad (\gamma,\eta)\mapsto
  \gamma\eta.
\]

The groupoid~\(G\) above is called \emph{topological} if \(G\) also
carries a topology such that the multiplication and inversion are
continuous mappings; for the continuity of the multiplication the
fibre product~\(G\times_{s, \base[G], r} G\) is equipped the subspace
topology of~\(G\times G\). As consequence, \(\mathrm{inv}\) is a
homeomorphism and \(r,s\colon G \rightrightarrows \base[G]\) are
continuous maps when \(\base[G]\) is given the subspace topology.

A topological groupoid~\(G\) is called \emph{locally compact} if its
topology is locally compact and its space of unit is a Hausdorff
subspace.

For \(u\in \base\), \(G^u,G_u,\) and \(G_u^u\) have their standard
meanings, namely,
\[
  G_u= s\inverse(\{u\}), G^u= r\inverse(\{u\}) \text{ and } G_u^u=
  G_u\cap G^u.
\]
Note that \(G_u^u\) is the isotropy group at~\(u\in \base[G]\).

The symbols \(\homeo\) and \(\iso\) will stand for `homeomorphic' and
`isomorphic' respectively.

\begin{example}
  \label{exa:gpd-of-trivial-equi-rel}
  Given a topological space \(X\), the product space \(X\times X\) has
  the following groupoid structure: \((x,y), (z,w) \in X\times X\) are
  composable if and only if \(y=z\) and the composition is given by
  \((x,y)(y,w) = (x,w) \); the inverse map is given by
  \((x,y)\inverse = (y,x)\).  The space of units of this groupoid is
  the diagonal in~\(X\) which we often identify with~\(X\). This
  groupoid is called the \emph{groupoid of the trivial equivalence
    relation}.
\end{example}

\begin{example}
  \label{exa:trans-gpd}
  Let \(G\) be a topological group acting continuously on right of a
  space~\(X\). The transformation groupoid for this action, denoted
  by~\(X\rtimes G\), is defined as follows: the underlying space of
  this groupoid is the cartesian product~\(X\times G\); two elements
  \((x,g)\) and \((y,t)\) are composable if and only if
  \(y= x\cdot g\) and the composition is given by
  \((x,g)(y,t) = (x, gt)\); the inverse map is given by
  \((x,g)^{-1} = (x\cdot g, g^{-1})\). Finally, if \(e\in G\) is the
  unit, then the space of units of this groupoid is \(X\times \{e\}\)
  which we identify with~\(X\).

  For a left \(G\) space \(X\), the transformation groupoid is defined
  similarly and is denoted by~\(G\ltimes X\).
\end{example}

A map of spaces \(f\colon X\to Y\) is called a \emph{local
  homeomorphism} if it is surjective and given \(x\in X\) has a
neighbourhood~\(U\subseteq X\) such that \(f(U)\subseteq Y\) is open
and \(f|_U\colon U\to f(U)\) is a homeomorphism. A local homeomorphism
is an open mapping. A covering map is a local homeomorphism.

\begin{definition}[Locally trivial groupoid~{\cite[page
    143]{Ehresmann1958Categories-Top-et-categories-diff}}]
  \label{def:loc-triv-gpd}
  A topological groupoid~\(G\) is called locally trivial if for each
  unit~\(u\), the restriction of the source map~\(G^u \to \base[G]\)
  is a local homeomorphism.
\end{definition}

\noindent Definition~\ref{def:loc-triv-gpd} can be equivalently stated
as \(r|_{G_u}\colon G_u \to \base[G]\) is a local homeomorphism for
every unit~\(u\).

\begin{definition}
  \label{def:etal-gpd}
  A topological groupoid is called \etale\ if its range (or,
  equivalently, source) map is a local homeomorphism.
\end{definition}

The transformation groupoid in Example~\ref{exa:trans-gpd} is \etale\
\emph{iff}~\(G\) is discrete.  \medskip

\subsection{Actions of groups and groupoids}

Hereon, all groups or groupoids, and their actions are topological. We
find Lee's books~\cite{Lee2012Intro-to-smooth-manifolds-book}
and~\cite{Lee2011Intro-to-top-manifolds-book} very useful when for
group actions and quotients by these actions. We also find his
terminologies more relevant and hence we adopt them.

The goal of this section is to establish
Example~\ref{exa:fund-gp-action} which will be used to discuss the
point-set topology on the fundamental groupoid. This example is
essentially Example~2.4
in~\cite{Muhly-Renault-Williams1987Gpd-equivalence} which the authors
attribute to Reinhart's~Proposition 2.37
in~\cite{Reinhart1983Folliations-book}. In this proposition, Reinhart
shows that the quotient map
\(\tilde{X}\times \tilde{X} \to (\tilde{X}\times \tilde{X})/
\FGp(X)\)---for the diagonal action of~\(\FGp(X)\) on the
product~\(\tilde{X}\times \tilde{X}\) of the universal covering space
of~\(X\)---is a covering map provided that \(X\) is locally path
connected and semilocally simply connected.

A point to notice here is that Reinhart shows that the quotient
space~\((\tilde{X}\times \tilde{X})/ \FGp(X)\) is well-defined using
covering space theory. However,
in~\cite[Example~2.4]{Muhly-Renault-Williams1987Gpd-equivalence}, one
needs to show that the deck transformation action of the fundamental
group on the simply connected covering space is proper. Details of
properness of this action seem to be missing---we give this
justification here.

In this attempt, we observed that our techniques need~\(X\) to be
Hausdorff; we do not know the proof when~\(X\) is not
Hausdorff. Nonetheless, for our purpose, Hausdorffness of~\(X\) is
necessary, as, even for a locally compact (locally Hausdorff
groupoid), we demand the space of units to be Hausdorff.

A \emph{(left) action} of a topological groupoid~\(G\) on a space
\(X\) is a pair \((m, a)\) where \(m\colon X\to \base[G] \) is an open
surjection, and \(a\colon G\times_{s, \base[G], m} X \to X \) is
function with following properties:
\begin{enumerate}
\item for any unit \(u\in \base\) and \(x\in m\inverse(u)\),
  \(a(u,x) = x\);
\item if \((\gamma,\eta) \in G\times G\) is a composable pair, and
  \((\eta, x)\in G\times_{s, \base[G], m} X\), then
  \((\gamma, a(\eta, x))\) is also in the fibre
  product~\(G\times_{s, \base[G], m} X \).  Moreover,
  \(a(\gamma, a(\eta, x)) = a(\gamma \eta, x)\).
\end{enumerate}

\noindent Here \( G\times_{s, \base[G], m} X \) is the fibre
product~\(\{ (\gamma, x)\in G\times X : s(\gamma) = m(x)\}\). We
call~\(m\) the momentum map of the action. A \emph{right action} is
defined similarly. Now on we shall simply write \(\gamma x\) or
\(\gamma\cdot x\) instead of~\(a(\gamma, x)\) for
\((\gamma, x)\in G\times_{s, \base[G], m} X\).  Nearly all the time we
abuse notation and language by ignoring the momentum and action maps,
and simply call \(X\) a left~\(G\)\nb-space.  When we say that \(X\)
is a left (or right) \(G\)\nb-space, we tacitly assume that the
momentum map for the action is denoted by~\(r_X\) (respectively,
\(s_X\)). This convention shall not cause any confusion in this
article. We assume that the momentum map of any action is an open
surjection.

\medskip

A map of spaces \(f\colon X\to Y\) is called \emph{proper} if the
inverse image of a compact set in~\(Y\) under~\(f\) is compact.

\begin{definition}[Free and proper actions]
  \label{def:proper-act}
  Let \(G\) be a topological groupoid acting on the left of a
  space~\(X\).
  \begin{enumerate}
  \item We say this action is \emph{proper} if the map
    \(a\colon G\times_{s, \base[G], r_X} X \to X\times X\) given by
    \((\gamma,x) \mapsto (\gamma x, x)\) is proper.
  \item The action is called \emph{free} if the map~\(a\) is
    one-to-one.
  \end{enumerate}
\end{definition}

\noindent Definition~\ref{def:proper-act}(1) is equivalent to saying
that for given compact sets \(K, L\subseteq X\) the set
\(\{\gamma\in G : (\gamma K) \cap L \neq \emptyset\}\) is compact
in~\(G\). Definition~\ref{def:proper-act}(2) is equivalent to saying
that the stabiliser of the action is trivial at every point of~\(X\).

Assume that \(G\) is a groupoid with open range map, and it acts
properly on a space~\(X\). Then, it is well-known that, the quotient
mapping~\(X\to G\7 X\) is open. Moreover, if~\(X\) is locally compact
(or Hausdorff or second countable), then so is the quotient~\(G\7 X\).

We assume that reader is familiar with continuous and proper actions
of a topological groupoids, see, for
example~\cite{Tu2004NonHausdorff-gpd-proper-actions-and-K}.

\begin{definition}[Covering space action~{\cite[Chapter 21,
    \S~Covering Manifolds]{Lee2012Intro-to-smooth-manifolds-book}}]
  \label{def:cov-space-act}
  Suppose a discrete group~\(G\) acts on a space~\(X\). The action is
  called a \emph{covering space action} if given any point~\(p\in X\)
  has a neighbourhood~\(U\) such that \(\gamma U\cap U = \emptyset\)
  for any \(\gamma \in G\) \emph{unless}~\(\gamma\) is the unit
  of~\(G\).
\end{definition}

It is a classical
result,~\cite[Theorem~12.14]{Lee2011Intro-to-top-manifolds-book} that
the action of a subgroup~\(G\) of homeomorphisms of a nice space~\(X\)
on the space is \emph{covering space action} \emph{iff} the quotient
map~\(X\to X/G\) is a normal covering map with~\(G\) being the group
of automorphisms of this covering space.

\begin{propdef}[Hausdorff criterion for the
  quotient,~{\cite[Proposition
    12.21]{Lee2011Intro-to-top-manifolds-book}}]
  \label{pdef:hausdorff-crit}
  Suppose a group~\(G\) acts on a topological (not necessarily
  Hausdorff) space~\(X\). Then following two statements are
  equivalent.
  \begin{enumerate}
  \item The quotient space~\(X/G\) is Hausdorff.
  \item For \(p,q\in X\) which are not in the same orbits, there are
    neighbourhoods~\(V\) and~\(W\) of~\(p\) and~\(q\), respectively,
    such that \((\gamma V) \cap W =\emptyset\) for
    all~\(\gamma\in G\).
  \end{enumerate}
  We call any of above properties the \emph{Hausdorff criterion for
    quotient spaces}.
\end{propdef}

\begin{lemma}
  \label{lem:prop-dicnt-act-are-proper}
  Let~\(G\) be a discrete group acting freely on a topological
  space~\(X\). Assume that the action is a covering space action and
  satisfies the Hausdorff criterion for the quotient. Then the action
  is proper.
\end{lemma}
\begin{proof}
  Proof of converse part of Lemma 21.11
  in~\cite{Lee2012Intro-to-smooth-manifolds-book} works here
  \emph{verbatim} after replacing the sequences in the proof by nets.
\end{proof}

Last lemma is essentially a weaker version of Lemma 21.11
in~\cite{Lee2012Intro-to-smooth-manifolds-book}; although Lemma 21.11
in~\cite{Lee2012Intro-to-smooth-manifolds-book} is stated for a Lie
group action on a manifold, it is not hard to see that the lemma holds
for an action of a topological group on a \emph{locally compact
  Hausdorff space}. However, if one removes the locally compactness
hypothesis, then the proof of one implication in the lemma does not
work anymore. But, as noticed in
Lemma~\ref{lem:prop-dicnt-act-are-proper}, the other implication
holds, and this is sufficient for our purpose.

\begin{example}
  \label{exa:fund-gp-action}
  Let \(X\) be path connected, locally path connected and semilocally
  simply connected space which is also \emph{Hausdorff}. Equip the
  fundamental group~\(\FGp(X)\) with the discrete topology (or see
  Corollary~\ref{cor:subspace-top-on:fibre-UC-top}).  Then the action
  of the fundamental group~\(\FGp(X)\) on the simply connected
  covering space~\(\tilde{X}\) by the deck transformations is
  free. Next, the quotient space~\(\tilde{X}/\FGp(X) = X\) is
  Hausdorff. Thus the deck transformation action satisfies the
  hypothesis of Lemma~\ref{lem:prop-dicnt-act-are-proper}. Therefore,
  the action is proper.
\end{example}

An immediate consequence of above examples is following: consider the
diagonal action of~\(\FGp(X)\) on~\(\tilde{X}\times \tilde{X}\), that
is, for \((x,y)\in \tilde{X}\times \tilde{X}\)
and~\(\gamma \in \FGp(X)\), \((x,y)\gamma = (x\gamma,y\gamma)\). Since
the original action is proper, the diagonal action is also
proper. Therefore, the
quotient~\((\tilde{X}\times \tilde{X})/\FGp(X)\) inherits the good
topological properties---such as Hausdorffness, locally compactness
and second countability---from~\(\tilde{X}\).

\subsection{Topology}
Our reference for Algebraic Topology is Hatcher's
book~\cite{Hatcher2002Alg-Top-Book}, except for the actions of
fundamental group as mentioned earlier. By the universal covering
space we mean the universal object in the category of connected
covering spaces, see~Spanier~\cite{Spanier1966Alg-Top-Book}*{page
  80}). As Massey shows in his book~\cite[Corollary 6.2,
Chapter~V]{Massey1991Basic-Course-in-Alg-Top}, a \emph{simply
  connected} covering space is the universal covering space. It is a
standard texbook result that a simply connected covering exists for a
path connected, locally path connected and semilocally simply
connected space.

We denote the unit closed interval \([0,1]\subseteq \R\) by
\(\UInt\). For a topological space~\(X\), \(\PathS{X}\) denotes the
set of paths in \(X\). Differently speaking,
\(\PathS{X}=\Cont(\UInt, X)\), the set of all continuous
functions~\(\UInt\to X\). For \(\gamma \in \PathS{X}\), we
call~\(\gamma(0)\) and~\(\gamma(1)\), respectively, the starting and
ending points of~\(\gamma\). Alternatively, we say that~\(\gamma\)
starts at \(\gamma(0)\) and ends at~\(\gamma(1)\). For \(x_0\in X\),
the constant function~\(x_0\) in~\(\PathS{X}\) is called the constant
path at~\(x_0\), and is denoted by~\(e_{x_0}\).

Continuing further, for \(x\in X\), \(\PathS{X}^x\) and
\(\PathS{X}_x\) denote, respectively, the set of paths in~\(X\) which
end at~\(x\) and the set of paths in~\(X\) which start at~\(x\).

Let \(\gamma,\eta\in\PathS{X}\). If \(\gamma(0)=\eta(1)\) we call the
pair~\((\gamma, \eta)\) composable (concatenable for topologists) and
\(\gamma \oblong \eta\) denotes the concatenated path. In case of
groupoids we adopt the convention that the arrows are from \emph{right
  to left}; the current choice of path concatenation is made to
maintain this consistency in the fundamental groupoid. For
path~\(\gamma\) above, \(\gamma^-\) is its \emph{opposite} path given
by \(\gamma^-(t)=\gamma(1-t)\). A path \(\gamma\in \PathS{X}\) is
called a loop if \(\gamma(0)=\gamma(1)\).

For \(X\) as above, let \(x\in X\). Then \(\FGd(X)\) and \(\FGp(X,x)\)
denote the fundamental groupoid of~\(X\) and fundamental group
of~\(X\) based at \(x\), respectively. Thus \(\FGd(X)\) is the
quotient of~\(\PathS{X}=\Cont(\UInt, X)\) by the \emph{endpoint
  fixing} path homotopy equivalence relation; the relation is denoted
by~\(\sim_p\). This quotient map need not be open in
general~\cite[Example~1, \S~22]{Munkress1975Topology-book}. Similarly,
\(\FGp(X,x)\) is a quotient of \(\PathS{X}_x^x\).  The (endpoint
fixing) homotopy class \(\gamma\in\PathS{X}\), is denoted
by~\([\gamma]\).

If \(X\) is path connected, we ignore the basepoint and simply
write~\(\FGp(X)\). In this case, we also realise the fundamental group
as the group of deck transformations of the universal covering space.

After setting the notation, next we define some basic notions. We do
it explicitly for the sake of transparency.

For \(X\) as earlier, a set \(V\subseteq X\) is called
\emph{relatively inessential} in \(X\), if the homomorphism
\(\FGp(V,x)\to\FGp(X,x)\) induced by the inclusion
\((V,x)\hookrightarrow (X,x)\) is trivial where \(x\in V\). The space
\(X\) is called
\begin{enumerate}[(i), leftmargin=*]
\item \emph{locally path connected} if for a given point \(x\in X\)
  and a neighbourhood \(V\) of it, there is an open path connected
  neighbourhood \(U\subseteq V\) of \(x\);
\item \emph{semilocally simply connected} if each \(x\in X\) has a
  relatively inessential neighbourhood.
\end{enumerate}

A covering map, covering space, an evenly covered neighbourhood and
other standard terms regarding covering spaces and simply connected
cover have their usual meanings as in
Hatcher~\cite{Hatcher2002Alg-Top-Book}.

\subsection{The compact-open topology}
\label{sec:comp-open-topol}

Let \(X\) and \(Y\) be spaces and \(\Cont(Y,X)\) denote the set of
continuous functions from \(Y\) to \(X\). For \(U\subseteq Y\) and
\(V\subseteq X\) define the following subset of \(\Cont(Y, X)\):
\[
  \Basic(U,V)=\{f\in \Cont(Y,X): f(U)\subseteq V\}.
\]

Recall from~\cite{Dugunji-Book}, that the sets of form
\(\Basic(K,U)\), where~\(K\subseteq Y\) is compact
and~\(U\subseteq X\) is open, constitute a subbasis for the
compact-open topology on \(\Cont(X,Y)\). Next we recall some
properties of the compact-open topology on~\(\Cont(X,Y)\) and their
consequences for the path space~\(\PathS{X}\).

\begin{property}\label{prop:open-set-prop}
  The above subbasic open sets of the compact-open topology have
  following properties:
  \begin{equation}\label{eq:CO-rel}
    \left\{\quad
      \begin{aligned}
        \Basic(A,W) &\subseteq \Basic(A,W') && \text{ if } W\subseteq W';\\
        \cap_{i=1}^n\Basic(A_i,W_i) &\subseteq \Basic(\cup_{i=1}^n
        A_i, \cup_{i=1}^n W_i);
      \end{aligned}
    \right.
  \end{equation}
  these relations are easy to check, see~\cite{Dugunji-Book}*{Chapter
    XII, Example 2}.
\end{property}

\begin{property}
  \label{prop:homeo-CO}
  Let \(\iota\colon X \to X'\) be a homeomorphism. Then it induces a
  homeomorphism \(I\colon \Cont(X',Y) \to \Cont(X,Y)\) by
  \(I(f) = f\circ \iota\) for \(f\in \Cont(X',Y)\). It maps the
  subbasic open set \(\Basic(A,W)\subseteq \Cont(X',Y)\) to the
  subbasic open set
  \(\Basic(\iota\inverse(A), W)\subseteq \Cont(X,Y)\).
\end{property}

\begin{property}\label{prop:basis-PX}
  When equipped with the compact-open topology \(\PathS{X}\) has a
  basis consisting of the sets of the form
  \(\cap_{i=1}^{n}\Basic(I_i, U_i)\) where \(n\in \N\),
  \(I_1,\dots,I_n\) are non overlapping closed intervals which cover
  \(\UInt\) and \(U_i\subseteq X\) is open for each \(1\leq i\leq n\)
  (\cite{Hatcher2002Alg-Top-Book}*{the Appendix on the compact-open
    topology}). To summarise, the sets of the form~\(\Basic(I, U)\),
  where \(I\subseteq \UInt\) is a closed interval (or a closed
  interval with rational endpoints) and \(U\subseteq X\) is an open
  set, yield a subbasis of~\(\PathS{X}\).
\end{property}

\begin{property}\label{prop:CO-Hausdorff}
  The function space \(\Cont(Y,X)\) is Hausdorff (or regular) in the
  compact-open topology if and only if \(X\) is Hausdorff (regular,
  respectively)~\cite{Dugunji-Book}*{Chapter XII, 1.3}.  But if \(Y\)
  is normal, \(\Cont(X,Y)\) need not be normal,
  see~\cite{Dugunji-Book}*{Example 5}.
\end{property}

\begin{property}\label{prop:CO-for-metric-sp}
  If \(Y\) is compact and \(X\) is a metric space, then the compact
  open topology on \(\Cont(Y,X)\) is same as the uniform metric
  topology~\cite{Dugunji-Book}*{Chapter XII, 8.2 (3)}.
\end{property}

\begin{property}\label{prop:CO-evaluation-map}
  The evaluation map,
  \[
    \Cont(Y,X)\times Y \to X, \quad (f,y) \mapsto f(y)
  \]
  where \((f,y)\in \Cont(Y,X)\times Y\), is continuous if \(Y\) is
  locally compact~\cite{Dugunji-Book}*{Chapter XII, Theorem 2.4
    (2)}. In particular, the evaluation maps associated with the path
  space \(\PathS{X}\) is continuous.
\end{property}

\begin{property}\label{prop:CO-sec-ctble}
  Assume that \(Y\) is locally compact Hausdorff second
  countable. Then \(\Cont(Y,X)\) is second countable in the
  compact-open topology \emph{iff} \(X\) is second
  countable~\cite[Theorem 3.7]{Edwards1999CO-topology}.  In
  particular, a path space \(\PathS{X}\) is second countable in the
  compact-open topology if and only if \(X\) is second countable.
\end{property}

\begin{property}\label{prop:CO-paracpt}
  If \(Y\) and \(X\) are metric spaces with~\(Y\) separable, then
  \(\Cont(Y,X)\) is paracompact in the compact-open
  topology~\cite{OMeara1971CO-top-paracompact}. As a consequence, a
  path space of a metric space is paracompact.
\end{property}

\smallskip

Finally, a remark is that the local compactness of the compact-open
topology seems a rare or restrictive phenomenon. For instance,
let~\(Y\) is a compact Hausdorff space. Then \(\Cont(Y, \C)\) is a
(complex) normed linear (in fact, a Banach) space, as a consequence of
Property~\ref{prop:CO-for-metric-sp}. Therefore, \(\Cont(X, \C)\) is
locally compact \emph{iff} \(X\) is finite.

\begin{lemma}
  \label{lem:subcov-CO-top}
  \begin{enumerate}[leftmargin=*]
  \item Assume that \(\mathcal{U}\) is a basis for the topology of a
    space \(X\). Let \(\mathcal{R}\) be the collection of subsets of
    \(\PathS{X}\) of the form \(\Basic(I,U)\) where
    \(I\subseteq \UInt\) is a closed interval and
    \(U\in \mathcal{U}\). Then \(\mathcal{R}\) forms a subbasis for
    the compact-open topology on \(\PathS{X}\).
  \item In particular, if the space \(X\) above is locally path
    connected and semilocally simply connected, then the sets of the
    form \(\Basic(I,U)\) with \(I\subseteq \UInt\) is a closed
    interval and \(U\subseteq X\) is a path connected and relatively
    inessential open neighbourhood forms a subbasis for the
    compact-open topology on \(\PathS{X}\).
  \item In particular, if the space \(X\) above is locally path
    connected and semilocally simply connected, then the sets of the
    form \(\Basic(I,U)\) with \(I\subseteq \UInt\) is a closed
    interval with \emph{rational endpoints} and \(U\subseteq X\) is a
    path connected and relatively inessential open neighbourhood forms
    a subbasis for the compact-open topology on \(\PathS{X}\).
  \end{enumerate}
\end{lemma}

\noindent The proof of the first part of the lemma uses following easy
observation which we leave for the reader to prove: Suppose that
\(\mathcal{S}_i\) is subbasis for a topology~\(\tau_i\) on a set~\(X\)
where \(i=1,2\). Assume that for given \(U_2\in \mathcal{S}_2\)
and~\(x\in X\), there is a basic \(\tau_1\)\nb-open set~\(B_1\) such
that \(x\in B_1\subseteq U_2\). Then \(\tau_2\subseteq \tau_1\).

\begin{proof}[Proof of Lemma~\ref{lem:subcov-CO-top}]
  
  \noindent (1): Every set in \(\mathcal{R}\) is clearly open in
  \(\PathS{X}\) in the compact-open topology. Let \(\Basic(I,V)\) be a
  subbasic open set in \(\PathS{X}\) where
  \(I\defeq [a,b]\subseteq \UInt\) is an interval, and
  \(V\subseteq X\) open, cf. Property~\ref{prop:basis-PX}. For given
  \(\alpha\in \Basic(I,V)\), we construct finitely many sets
  \(A_1,\dots, A_k\in\mathcal{R}\) with
  \(\alpha\in\cap_{j=1}^k A_j\subseteq \Basic(I,V)\) which is the
  claim of the lemma.

  Since~\(\mathcal{U}\) is a basis, we cover
  \(V= \cup_{\theta\in J} U_\theta \) by open sets~\( U_\theta\)
  in~\(\mathcal{U}\) where~\(J\) is some indexing set. Then
  \(\{\alpha\inverse(U_\theta)\}_{\theta\in J}\) is an open cover of
  \(I=[a,b]\) in \(\UInt\).  Using a Lebesgue number for the open
  cover~\(\{\alpha\inverse(U_\theta)\}\), get a finite sequence
  \(a_0=a<a_1<\dots<a_l=b\) in \([a,b]\) such that for each
  \(0\leq j \leq l\), there an index \(\theta_j\in J\) so that
  \([a_j,a_{j+1}]\) contained in
  \(\alpha\inverse(U_{\theta_j})\). Thus
  \(\alpha\in \cap_{j=0}^l \Basic([a_j,a_{j+1}],U_{\theta_j})\). The
  second relations in Equation~\eqref{eq:CO-rel} implies that
  \[
    \cap_{j=0}^l \Basic([a_j,a_{j+1}],U_{\theta_j})\subseteq
    \Basic(\cup_{j=0}^l [a_j,a_{j+1}],\cup_{j=0}^l
    U_{\theta_j})\subseteq \Basic(I,V).
  \]
  \smallskip
 
  \noindent (2) and (3): These follows directly from~{(i)} as the path
  connected relatively inessential neighbourhoods form a subbasis for
  the topology on \(X\).
\end{proof}

\begin{remark}
  \label{rem:countable-basis-of-CO-2}
  Lemma~\ref{lem:subcov-CO-top}(3) implies that, in this case,
  \(\PathS{X}\) is second countable. Therefore, a space which is
  second countable, locally path connected and semilocally simply
  connected has a countable basis generated by countable subbasic open
  set of form~\(\Basic(I, U_n)\) where \(I\) is a closed subinteraval
  of~\(\UInt\) with rational endpoints and \(U_n\subseteq X\) is one
  of the countable basic open set that is path connected and
  semilocally simply connected.
\end{remark}

\section{Fundamental groupoid as a topological groupoid}
\label{sec:fund-group-as}

\subsection{The two topologies}
\label{sec:two-topologies}

In this section, we describe the CO' and UC topologies on the
fundamental groupoid, and show that they are same for a locally path
connected and semilocally simply connected space. We also mention an
example when the equality fails to hold due to the lack of semilocally
simply connectedness.

\begin{definition}[CO' topology]
  \label{def:CO-top}
  Let \(X\) be a space. The CO' topology on \(\FGd(X)\) is the
  quotient topology induced by the compact-open topology on
  \(\PathS{X}\) via the quotient map \(q\colon \PathS{X}\to \FGd(X)\)
  that sends a path to its path homotopy class.
\end{definition}

Next, for defining the UC topology, we assume that~\(X\) is locally
path connected and semilocally simply connected.  For
\(\gamma\in\PathS{X}\) and neighbourhoods \(\gamma(1)\in U\) and
\(\gamma(0)\in V\) in \(X\), define the following subsets of
\(\PathS{X}\) and \(\FGd(X)\), respectively,
\noindent
\begin{align*}
  \widetilde{N}(\gamma,U,V) &=\{\delta\oblong \gamma \oblong \theta
                              :\delta\in\PathS{U} \text{ with }
                              \delta(0)=\gamma(1), \text{ and } \theta\in\PathS{V} \text{ with } 
                              \gamma(0)=\theta(1)\}\\
  N([\gamma],U,V) &=\{[\delta\oblong \gamma \oblong \theta]
                    :\delta\in\PathS{U} \text{ with }
                    \delta(0)=\gamma(1), \text{ and } \theta\in\PathS{V} \text{ with } 
                    \gamma(0)=\theta(1)\}.
\end{align*}
                
Thus, in short \(N([\gamma],U,V) =q(\tilde{N}(\gamma,U,V))\). As a
consequence of Lemma~2 in~\cite{Brown-Danesh-1975-Top-FG-1}, the sets
of form \(N([\gamma], U,V)\), where \(U\) and~\(V\) are path connected
and semilocally simply connected, form a basis for the a topology
on~\(\FGd(X)\).

\begin{definition}[UC topology]
  \label{def:UC-top}
  Let \(X\) be a path connected, locally path connected and
  semilocally simply connected space. The topology on~\(\FGd(X)\)
  generated by the basic sets \(N([\gamma],U,V)\), where
  \([\gamma]\in\FGd(X)\) and \(U,V\subseteq X\) are path connected
  relatively inessential neighbourhoods of \(\gamma(1)\) and
  \(\gamma(0)\) respectively, is called the UC topology.
\end{definition}

Our next goal is Proposition~\ref{prop:CO-UC-tops-same} which we state
after the following two lemmas. We don't claim any originality for the
proof of this proposition as it is basically the proof of Lemma 2.1
in~\cite{Fischer-Zastrow2007Gen-Uni-Coverings-and-Shape-Gp} written
for the \emph{both ends} of (the homotopy class of) a path. The
proposition uses Lemmas~\ref{lem:CO-coarser}
and~\ref{lem:support-for-UC-equlas-CO}.

\begin{lemma}\label{lem:CO-coarser}
  For a topological space \(X\), the CO' topology is coarser than the
  UC topology. That is \(\textnormal{CO'}\subseteq \textnormal{UC}\).
\end{lemma}

\begin{proof}
  Assume that \( \PathS{X}\) has compact open topology and
  \(q\colon \PathS{X} \to \FGd(X)\) is the quotient map. Let a
  CO'-open set \(M \subseteq \FGd(X)\) be given. That is,
  \(q^{-1}(M)\) is open in \(\PathS{X}\). Let \([\alpha] \in M\). Then
  we show that there is a UC-neighbourhood~\(N([\alpha], V^1, V^0)\)
  of~\([\alpha]\) that is contained in~\(M\) and open in the UC
  topology.

  Firstly note that since \(q^{-1}(M)\subseteq \PathS{X}\) is open in
  the compact-open topology, \(\alpha\) has a basic (compact-open)
  neighbourhood~\( \bigcap_{i=1}^{n} \Basic(I_i,U_i)\subseteq
  q^{-1}(M)\) here \(I_1,I_2,\cdots, I_n\) are closed intervals
  in~\([0,1]\) and \(U_1,U_2,\cdots,U_n\) are open subsets of \(X\)
  (cf. Lemma~\ref{lem:subcov-CO-top}(2)).

  Now we start constructing the required UC-neighbourhood: for
  \(i=0,1\), let \(S_i\) be the set of indices~\(k\in \{1,\dots, n\}\)
  for which \(i\in I_k\).  Set \(V^i = \cap_{k\in S_i} U_k\) for
  \(i= 0, 1\). Then \(\alpha(0) \in V^0\) and \(\alpha(1) \in V^1\).
  The set \(N([\alpha], V^1,V^0) \subseteq M\) is the desired
  neighbourhood of~\([\alpha]\) in the UC topology.

  Thus we need to show that this neighbourhood is contained
  in~\(M\). For that assume \([\beta] \in N([\alpha], V^1,V^0)\) is
  given. That is, \(\beta = \theta \oblong \alpha \oblong \gamma \)
  for some paths \(\theta \in \PathS{V^1}\) and
  \(\gamma \in \PathS{V^0}\) appropriately concatinable
  with~\(\alpha\). Choose \(\epsilon \in (0,1)\) such that
  \(\alpha([0,\epsilon]) \subseteq V^0\) and
  \(I_i \cap [0,\epsilon] =\emptyset\) whenever \(0 \notin
  I_i\). Similarly, let \(\delta \in (0,1)\) such that
  \(\alpha([\delta,1]) \subseteq V^1\) and
  \(I_i \cap [\delta,1] =\emptyset\) whenever \(1 \notin I_i\). Due to
  this choice of \(\epsilon\) and \(\delta\), the path
  \(\beta = \theta \oblong \alpha \oblong \gamma\) looks as follows:
  \begin{enumerate}[(i),leftmargin=*]
  \item \(\beta([0,\epsilon])\subseteq V^0\);
  \item\label{it:look-2} \(\beta\) equals \(\alpha\) on
    \([\epsilon, \delta]\); That is, \(\alpha (t) = \beta(t)\) for
    \(t \in [\epsilon, \delta]\).
  \item \(\beta([\delta, 1]) \subseteq V^1\).
  \end{enumerate}
  
  Now we can show that \(\beta\in \bigcap_{i=1}^{n}
  \Basic(I_i,U_i)\). For this note that for given
  index~\(k\in \{1,\dots,n\}\), there are three choices for~\(I_k\),
  namely,
  \begin{enumerate}[(a)]
  \item\label{it:choice-1} \(0,1\notin I_k\),
  \item\label{it:choice-2} \(0\in I_k\) but \(1\notin I_k\),
  \item\label{it:choice-3} \(0\notin I_k\) but \(1\in I_k\),
  \item \label{it:choice-4} \(0,1\in I_k\).
  \end{enumerate}
  We note that for all these choices, \(\beta(I_k)\subseteq U_k\):
  Choice~\ref{it:choice-1} is equivalent to saying that
  \(k\notin S^i\) for \(i=0,1 \). Therefore,
  \(I_k\subseteq [\epsilon, \delta]\). Thus
  \(\beta(I_k) = \alpha(I_k) \subseteq U_k\).

  For Choice~\ref{it:choice-2}, we have
  \[
    \beta(t) =
    \begin{cases}
      \alpha(t) \in U_k, & \quad \textup{ for } \epsilon \leq t,\\
      \beta(t) \in V^0 \subseteq U_k, & \quad \textup{ for } 0 \leq t
      \leq \epsilon.
    \end{cases}
  \]
  
  \noindent The first case in last equation holds
  because~\(\alpha\in \cap_{i=1}^n \Basic(I_i,U_i)\).  Similar to the
  last case, one can see that for Choice~\ref{it:choice-3},
  \(\beta(I_k) \subseteq U_k\).

  Finally, Choice~\ref{it:choice-4} forces \(I_k=\UInt\) and then
  \[
    \beta(t) =
    \begin{cases}
      \alpha(t) \in U_k,                    & \quad \textup{ when } \epsilon \leq t \leq \delta,                                       \\
      \beta(t) \in V^0 \subseteq U_k,       & \quad \textup{ when } 0 \leq t \leq \epsilon,                                            \\
      \beta(t) \subseteq V^1 \subseteq U_k, & \quad \textup{ when }
      \delta \leq t\leq 1.
    \end{cases}
  \]
  Thus \(\beta \in \bigcap_{i=1}^{n} \Basic(I_i,U_i)\) which proves
  the lemma.
\end{proof}

Can the CO' topology be strictly coarser than UC topology? The answer
yes. See Remark~\ref{rem:CO-strict-coarser} for details.

\begin{proposition}
  \label{prop:CO-UC-tops-same}
  For a locally path connected and semilocally simply connected
  space~\(X\), the CO' and UC topologies on \(\FGd(X)\) are same.
\end{proposition}

\begin{figure}[hbt]
  \centering \scalebox{0.8}{
\begingroup%
  \makeatletter%
  \providecommand\color[2][]{%
    \errmessage{(Inkscape) Color is used for the text in Inkscape, but the package 'color.sty' is not loaded}%
    \renewcommand\color[2][]{}%
  }%
  \providecommand\transparent[1]{%
    \errmessage{(Inkscape) Transparency is used (non-zero) for the text in Inkscape, but the package 'transparent.sty' is not loaded}%
    \renewcommand\transparent[1]{}%
  }%
  \providecommand\rotatebox[2]{#2}%
  \newcommand*\fsize{\dimexpr\f@size pt\relax}%
  \newcommand*\lineheight[1]{\fontsize{\fsize}{#1\fsize}\selectfont}%
  \ifx\svgwidth\undefined%
    \setlength{\unitlength}{273.15827087bp}%
    \ifx\svgscale\undefined%
      \relax%
    \else%
      \setlength{\unitlength}{\unitlength * \real{\svgscale}}%
    \fi%
  \else%
    \setlength{\unitlength}{\svgwidth}%
  \fi%
  \global\let\svgwidth\undefined%
  \global\let\svgscale\undefined%
  \makeatother%
  \begin{picture}(1,0.30596649)%
    \lineheight{1}%
    \setlength\tabcolsep{0pt}%
    \put(0,0){\includegraphics[width=\unitlength,page=1]{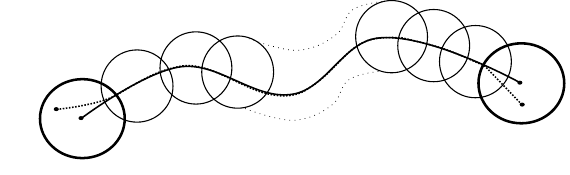}}%
    \put(-0.00277069,0.00645302){\makebox(0,0)[lt]{\lineheight{1.25}\smash{\begin{tabular}[t]{l}$U^0$\end{tabular}}}}%
    \put(0.85115527,0.04705388){\makebox(0,0)[lt]{\lineheight{1.25}\smash{\begin{tabular}[t]{l}$U^1$\end{tabular}}}}%
  \end{picture}%
\endgroup%
}
  \caption{In the above figure, $\alpha$ is denoted by the thick path
    and $\beta $ is the dotted path.}\label{fig:CO-UC-tops-same}
\end{figure}
  
\begin{proof}
  CO' is coarser than UC follows from Lemma~\ref{lem:CO-coarser}.  For
  the converse, let \([\alpha] \in \FGd(X)\) and let
  \(N([\alpha], U^1, U^0)\) be a basic UC\nb-open neighbourhood of it
  wherein \(U^0\) and \(U^1\) are path connected and relatively
  inessential, see Lemma~\ref{lem:subcov-CO-top}(2). We need to show
  that \(q\inverse(N([\alpha], U^1, U^0)) \subseteq \PathS{X}\) is
  open in the compact-open topology. For this let
  \(\beta \in q\inverse(N([\alpha], U^1,U^0))\) be given. We construct
  a compact-open neighbourhood~\(T\) of \(\beta\) that is contained
  in~\(q\inverse(N([\alpha], U^1, U^0)\).

  Firstly, using the continuity of~\(\beta\), choose
  \(\epsilon,\delta \in (0,1)\) such that
  \begin{enumerate}[(i),leftmargin=*]
  \item \(\beta([0,\epsilon]) \subseteq U^0\);
  \item \(\beta(t) = \alpha (t)\) for \(\epsilon \leq t \leq \delta\);
  \item \(\beta([\delta, 1]) \subseteq U^1\).
  \end{enumerate}
  Cover the compact set \(\beta([\epsilon,\delta])\) by path connected
  relatively inessential open sets,
  cf.~Figure~\ref{fig:CO-UC-tops-same}.  Obtain the Lebesgue number
  associated to the open cover of \([\epsilon,\delta]\) that is
  inverse image under~\(\beta\) of the prior open cover.  Using this
  Lebesgue number, we get a partition
  \(\epsilon = t_1 < t_2< \cdots < t_n = \delta\) of
  \([\epsilon, \delta]\) such that \(\alpha([t_i,t_{i+1}])\) lies in
  some path connected open relatively inessential neighbourhood
  \(W_i\) of \(X\) for \(i=1,2,\cdots,n-1\). We call \(W_0= U^0\) and
  \(W_n = U^1\). And finally, let \(V_i\) be the path component of
  \(W_i\cap W_{i+1}\) which contains \(\alpha(t_{i+1})\) for
  \(i=0,1,2,\cdots,n-1\). Now we define out desired set
  \(T\subseteq \PathS{X}\) as follows:
  
  \begin{multline}\label{eq:comp-of-top-basic-nbd}
    T = \Basic([0,\epsilon], W_0)\cap \Basic(\{t_1\}, V_0) \cap
    \Basic([t_1,t_2],W_1) \cap \Basic(\{t_2\},V_1) \cap
    \Basic([t_2,t_3], W_2)\\ \cap \Basic(\{t_3\}, V_2) \cap \cdots
    \cap \Basic([t_{n-1},t_n],W_{n-1})\cap \Basic(\{t_n\}, V_{n-1})
    \cap \Basic([\delta, 1], W_{n}).
  \end{multline}
  By definition, \(T\) is a set that is open in the compact-open
  topology on~\(\PathS{X}\) and
  contains~\(\beta\). Lemma~\ref{lem:support-for-UC-equlas-CO} shows
  that given~\(\xi\in T\), then
  \([\xi] \in N([\alpha], U^1, U^0) \implies T\subseteq
  q\inverse(N([\alpha], U^1, U^0))\) which completes the proof.
\end{proof}

\begin{lemma}
  \label{lem:support-for-UC-equlas-CO}
  Let~\(T\subseteq \PathS{X}\) be as in the proof of
  Proposition~\ref{prop:CO-UC-tops-same} defined in
  Equation~\ref{eq:comp-of-top-basic-nbd}. Assume the same convention,
  hypothesis and notation as in that proof. Then for
  given~\(\xi\in T\), there are paths \(\sigma_0\in \PathS{W_0}\) and
  \(\sigma_n\in \PathS{W_n}\) with
  \([\xi]= [\sigma_n\oblong \alpha \oblong \sigma_0]\). In short,
  \([\xi]\in N([\alpha], U^1, U^0)\).
\end{lemma}
\begin{proof}
  Using the path connectedness of \(W_0\) and \(W_n\) choose paths
  \(\sigma_0\) from \(\xi(0)\) to \(\alpha(0)\) and \(\sigma_n\) from
  \(\alpha(1)\) to \(\xi(1)\). Since \(\alpha(t_i)\) and \(\xi(t_i)\)
  are in same path component \(V_i\), choose path \(\eta_i\) from
  \(\alpha(t_i)\) to \(\xi(t_i)\) for \(i=1,2,\cdots n\). Now we have:
  \begin{enumerate}[(a)]
  \item \label{item-eqiv-1}
    \( [\eta_1 \oblong \alpha|_{[0,t_1]}\oblong \sigma_0] =
    [\xi|_{[0,t_1]}]\), because \(W_0\) is relatively inessential;
  \item \label{item-eqiv-2}
    \( [\eta_{i+1} \oblong \alpha|_{[t_{i+1},t_i]}\oblong \eta^{-}_i]
    = [\xi|_{[t_{i+1},t_i]}]\), because \(W_i\) is relatively
    inessential for \(1\leq i \leq n-1\);
  \item \label{item-eqiv-3}
    \( [\sigma_n \oblong \alpha|_{[1,t_n]}\oblong \eta^{-}_n] =
    [\xi|_{[t_n,1]}]\), because \(W_n\) is relatively inessential.
  \end{enumerate}
  Using~\ref{item-eqiv-1}, \ref{item-eqiv-2} and \ref{item-eqiv-3} we
  have
  \begin{multline*}
    [\xi] = [\sigma_n \oblong \alpha|_{[1,t_n]}\oblong \eta^{-}_n]
    [\eta_{n} \oblong \alpha|_{[t_{n},t_{n-1}]}\oblong \eta^{-}_{n-1}]
    \cdots [\eta_2 \oblong \alpha|_{[t_2,t_1]}\oblong \eta^{-}_1]
    [\eta_1 \oblong \alpha|_{[0,t_1]}\oblong \sigma_0] \\
    = [\sigma_n \oblong \alpha|_{[1,t_n]} \oblong
    \alpha|_{[t_n,t_{n-1}]} \oblong \cdots \oblong \alpha|_{[0,t_1]}
    \oblong \sigma_0 ] = [\sigma_n \oblong \alpha \oblong \sigma_0];
  \end{multline*}
  which was the claim of the Lemma.
\end{proof}

\begin{corollary}[Corollary of Proposition~\ref{prop:CO-UC-tops-same}]
  \label{cor:quotient-from-PX-to-PiX-open}
  Let \(X\) be locally path connected semilocally simply connected
  space; equip \(\PathS{X}\) with the compact-open topology and
  \(\FGd(X)\) with the CO' topology. Then the quotient map
  \(q\colon \PathS{X}\to \FGd(X)\) is open.
\end{corollary}
\begin{proof}
  We use the fact that the CO' and UC topologies on~\(\FGd(X)\) are
  same to prove this. Let
  \(U=\cap_{j=1}^n\Basic(I_j,V_j)\subseteq \PathS{X}\) be a basic open
  set in the compact-open topology where \(I_j\) are closed intervals,
  and \(V_j\subseteq X\) are path connected relatively inessential
  neighbourhoods. Let \([\alpha]\in q(U)\) be given. We construct a
  UC-neighbourhood of~\([\alpha]\) contained in~\(q(U)\).  For that,
  we may assume that \(\alpha'(0)\in V_1\) and
  \(\alpha'(1)\in V_n\)---relabel the indices if required. Then, note
  that~~\(N([\alpha],V_1,V_n)\) is an open neighbourhood
  of~\([\alpha]\) contained in~\(q(U)\).
\end{proof}

\subsection{Topological fundamental groupoid}
\label{sec:topol-fund-group-1}

Assume that \(X\) is a locally path connected and semilocally simply
connected space. Then the UC and CO' topologies on~\(\FGd(X)\) are
same; in this section, we prove that~\(\FGd(X)\) is a topological
fundamental groupoid when equipped with any of these topologies. At
this point, one should recall that
Fabel~\cite{Fabel2011HE-Counterexample} points out that if~\(X\) is
not semilocally simply connected, then the fundamental group equipped
using the CO' topology may fail to be topological; in particular, the
multiplication may fail to be continuous. As a consequence, the
fundamental groupoid fails to be topological.

For the fundamental groupoid~\(\FGd(X)\), the fibre product
\(\FGd(X)\times_{s, \base[\FGd(X)], r} \FGd(X)\), which consists of
composable arrows, shall be denoted by~\({\FGd(X)}^{(2)}\).

\begin{theorem}\label{thm:fund-gpd-top}
  Let \(X\) be a locally path connected and semilocally simply
  connected space. Then the fundamental groupoid \(\FGd(X)\) is
  topological groupoid when equipped with UC or, equivalently, CO'
  topology.
\end{theorem}

\begin{proof}[Proof of Theorem~\ref{thm:fund-gpd-top}]
  \emph{Continuity of multiplication:} We use the definition of
  continuity as in Theorem~18.1(4) in
  Munkres~\cite{Munkress1975Topology-book}. Let
  \(([\alpha],[\beta])\in \FGd(X)^{(2)}\). Consider an open set \(W\)
  of \(\FGd(X)\) containing \([\alpha \oblong \beta]\). Take a a basic
  UC\nb-open
  neighbourhood~\(N([\alpha \oblong \beta], U, V)\subseteq W\) of
  \([\alpha\oblong \beta]\). Choose a relatively inessential open
  neighbourhood~\(Z\) of \(\alpha(0)= \beta(1)\). Consider the open
  set~\(O\) in \(\FGd(X)^{(2)}\) given by
  \[
    O= \big(N([\alpha], U, Z) \times N([\beta], Z, V) \big) \cap
    \FGd(X)^{(2)}.
  \]
  Then \(O\) is a neighbourhood of~\(([\alpha], [\beta])\) in
  \(\FGd(X)^{(2)}\).

  Our claim is that for any \(([\xi], [\eta])\in O\), the product
  \([\xi\oblong \eta] \in N([\alpha\oblong \beta, U, V])\). Validity
  of this claim proves the continuity of multiplication.
    
  Proof of the claim: since \(([\xi], [\eta])\in O\), let
  \([\xi] = [\theta_1 \oblong \alpha \oblong \delta_1]\) and
  \([\eta] = [\theta_2 \oblong \beta \oblong \delta_2]\) where
  \(\theta_1 \in \PathS{U}\); \(\delta_1, \theta_2 \in \PathS{Z}\);
  \(\delta_2 \in \PathS{V}\); and \(\theta_2(1) = \delta_1 (0)\). Now
  we have
    
  \begin{multline*}
    [\xi \oblong \eta] =  [\theta_1 \oblong \alpha \oblong \delta_1 \oblong \theta_2 \oblong\ \beta \oblong \delta_2] =  [\theta_1 \oblong \alpha] [ \delta_1 \oblong \theta_2 ] [ \beta \oblong \delta_2] =  [\theta_1 \oblong \alpha] [ \beta \oblong \delta_2]\\
    = [\theta_1 \oblong \alpha \oblong \beta \oblong \delta_2] \in
    N([\alpha\oblong \beta], U, V ).
  \end{multline*}

  \noindent The third equality of the first line above follows
  because~\(Z\) is relatively inessential.  \smallskip

  \emph{The continuity of the inverse map:} Let \(\textrm{inv}\) be
  the inverse map of~\(\FGd(X)\) that send an element to its
  inverse. Let \(\iota\colon \UInt \to \UInt\) be the affine
  homeomorphism \(t\mapsto 1-t\) for \(t\in \UInt\).  Due to
  Property~\ref{prop:homeo-CO}, this affine homeomorphism, induces a
  self-homeomorphism \(\hat{\iota} \colon \PathS{X}\to \PathS{X}\) of
  the function space equipped with the compact-open topology; the map
  is given by \(\hat{\iota}(\alpha) = \alpha\circ \iota\) for
  \(\alpha\in \PathS{X}\). This self-homeomorphism also preserves path
  homotopies, and hence, by the universal property of the quotient
  space, it induces a homeomorphism
  \([\hat{\iota}]\colon \FGd(X) \to \FGd(X)\). Now we note that
  \(\mathrm{inv} = [\hat{ \iota}]\).
\end{proof}

The following observation implies an alternate proof for the
continuity of the inversion map. The observation is that
\[
  \textrm{inv} \biggr(N([\alpha^{-1}], V,U) \biggr) = N([\alpha], U,
  V);
\]
for any basic UC-open neighbourhood~\(N([\alpha], U, V)\). Due to this
observation, the inversion in~\(\FGd(X)\) is a continuous map.

\medskip

Some standard observations are that the space~\(X\), the simply
connected covering space~\(\tilde{X}\) of it, and the fundamental
group~\(\FGp(X,x)\) for \(x\in X\) are subsets of~\(\FGd(X)\). In the
following discussion, we show that these are \emph{in fact}
subspaces. For the sake of completeness of discussion, here are some
observations:
\begin{enumerate}[(i), leftmargin=*]
\item\label{it:base-homeo} The space of units~\(\FGd(X)^{(0)}\)
  of~\(\FGd(X)\) consists of constant paths. Therefore, it is in
  one-to-one correspondence with~\(X\). Let
  \(\Upsilon \colon X \to \FGd(X)^{(0)}\) be this correspondence that
  sends a point in~\(X\) to the homotopy class of the corresponding
  constant path.

  Note that~\(\Upsilon\) is induces by the continuous mapping
  \(X\xrightarrow{\upsilon} \PathS{X}\) that maps \(x\in X\) the
  corresponding constant path in~\(x\); here \(\PathS{X}\) is equipped
  with the compact-open topology. The mapping~\(\upsilon\) is
  continuous because the inverse image of a subbasic open
  set~\(\Basic(I,V)\subseteq \PathS{X}\) is~\(V\).
\item\label{it:cov-space-homeo} For \([e_x]\in \FGd(X)^{(0)}\), where
  \(e_x\) is the constant path at \(x\in X\), the fibre over its range
  is
  \[
    \FGd(X)^{[e_x]} = \{[\gamma]\in \FGd(X): \gamma(1)=x\}.
  \]
  Recall the standard construction of the simply connected covering
  space~\(\tilde{X}\) of~\(X\), say
  from~\cite{Hatcher2002Alg-Top-Book}. Then, it is clear that
  \(\FGd(X)^{[e_x]}\) is in one-to-one correspondence
  with~\(\tilde{X}\). The covering map \(p\colon \tilde{X} \to X\) is
  identified with the restriction~\(s|_{\FGd(X)^{[e_x]}}\) of the
  source map.
\item\label{it:fgp-homeo} Finally, for \(x\in X\), the fundamental
  group \(\FGp(X,x)=\FGd(X)^{[e_x]}_{[e_x]}\).
\end{enumerate}
  
\noindent The next corollary shows the above bijections of sets are
actually homeomorphism of space.

\begin{corollary}
  \label{cor:subspace-top-on:fibre-UC-top}
  Let \(X\) be path connected, locally path connected and semilocally
  simply connected space.
  \begin{enumerate}
  \item The mapping~\(\Upsilon\) in Point~\ref{it:base-homeo} above is
    a homeomorphism of spaces.
  \item In Point~\ref{it:cov-space-homeo}, the subspace topology
    on~\(\FGd(X)^{e_x}\) is same as the covering space topology on it.
  \item With the subspace topology of~\(\FGd(X)\), the fundamental
    group of~\(X\) is a discrete group, cf.Point~\ref{it:fgp-homeo}
    above.
  \end{enumerate}
  
\end{corollary}
\begin{proof}
  (1): The bijection~\(\Upsilon\) is continuous as it is the composite
  of continuous functions
  \(X\xrightarrow{\upsilon} \PathS{X} \xrightarrow{q} \FGd(X)\). We
  show that~\(\Upsilon\) is also open. For that let \(U\subseteq X\)
  be a nonempty relatively inessential open set.  Fix~\(z\in U\). Then
  \(\Upsilon(U) = \FGd(X)^{(0)} \cap N([e_z], U, U)\). Thus
  \(\Upsilon(U)\) is open in the subspace topology
  of~\(\FGd(X)^{(0)}\).

  \noindent (2): This is clear if one looks at the proof of
  construction of the simply connected covering space of~\(X\). And
  the covering map \( \FGd(X)^{[e_x]}\to \FGd(X)^{(0)}\) is the
  restriction of the source map of the groupoid.

  \noindent (3): Due to~(1) and~(2), we can identify the restriction
  of the source map \(\FGd(X)^{[e_x]} \to \FGd(X)^{(0)}\) with the
  universal covering map~\(p \colon \tilde{X} \to X\) which is a local
  homeomorphism. As the fundamental
  group~\(\FGp(X,x) = p\inverse(\{x\})\), it is discrete.
\end{proof}

\noindent Reader may readily see that in (2)~of
Corollary~\ref{cor:subspace-top-on:fibre-UC-top}~\(r\inverse([e_x])\)
can be replaced by~\(s\inverse([e_x])\) with appropriate
modifications. If~\(X\) is~\(T^1\), then it follows from~(2) of last
lemma that the covering space is a closed subspace of~\(\FGp(X)\).

Using Corollary~\ref{cor:subspace-top-on:fibre-UC-top}, we abuse the
language and notation by saying the~\(X\) is the space of units
of~\(\FGd(X)\), and a fibre over a unit~\(x\in X\) under the range (or
source) map is the universal covering space~\(\tilde{X}\) of~\(X\)
at~\(x\). Also, while thinking of the fundamental group as topological
group, we shall consider it to be discrete.

\begin{observation}\label{obs:dynamics-map-loc-homeo}
  For the space~\(X\) in last
  Corollary~\ref{cor:subspace-top-on:fibre-UC-top}, consider the map
  \[
    r\times s \colon \FGd(X) \to X \times X; \quad r\times s([\gamma])
    = (\gamma(1), \gamma(0)) \quad \text{ for } [\gamma]\in \FGd(X).
  \]
  Take a basic UC-open neighbourhood \(N([\alpha], U, V)\) of
  \([\alpha] \in \FGd(X)\) wherein \(U\) and \(V\) are path connected
  relatively inessential neighbourhoods of~\(\gamma(1)\)
  and~\(\gamma(0)\), respectively. Then
  \((r\times s) (N([\alpha], U, V)) = U\times V\) where the latter set
  is a basic open set in~\(X\times X\). Thus the mapping~\(r\times s\)
  is a local homeomorphism. Since~\(X\) is path connected, this
  mapping is also surjective.
\end{observation}

\begin{proposition}\label{prop:gpd-of-simply-conn-space}
  Let \(X\) be a locally path connected and simply connected
  space. Then \(\phi \colon \FGd(X) \to X\times X\) by
  \(\phi([\gamma]) = (\gamma(1), \gamma(0))\) is an isomorphism of
  topological groupoids where the latter is the groupoid of the
  trivial equivalence on~\(X\).
\end{proposition}

\noindent Note that, by definition, a simply connected space is path
connected. And such a space is automatically semilocally simply
connected. Reader may recall the groupoid of trivial equivalence from
Example~\ref{exa:gpd-of-trivial-equi-rel}.

\begin{proof}[Proof of
  Proposition~\ref{prop:gpd-of-simply-conn-space}]
  
  Since \(X\) is simply connected, the map \(\phi\) defined in the
  proposition is bijective. It can be easily check that this is an
  isomorphisms of algebraic groupoids. Note that the map~\(\phi\) in
  above proposition is essentially the map~\(r\times s\) in
  Observation~\ref{obs:dynamics-map-loc-homeo}. Hence,~\(\phi\) is
  continuous. This remark also says that~\(\phi\) is a local
  homeomorphism, therefore, it is an open map.
\end{proof}

We end this subsection with some concluding remarks.

\begin{remark}
  \label{rem-Fgd-local-trivial-not-etale}
  (1) and~(2) in Corollary~\ref{cor:subspace-top-on:fibre-UC-top}
  clearly show that the fundamental groupoid is a locally trivial
  groupoid (Definition~\ref{def:loc-triv-gpd}) but not an {\etale}
  (Definition~\ref{def:etal-gpd}) unless the path connected space is a
  point. Thus the claim in Example~\cite[Example
  2.1.4]{Khalkhali2013Basic-NCG-Book} that a fundamental groupoid is
  {\etale} is not correct. At the same time,
  Observation~\ref{obs:dynamics-map-loc-homeo} describes the local
  triviality condition for the fundamental groupoid.
\end{remark}

\begin{remark}
  \label{rem:CO-strict-coarser}
  For general topological spaces, it is possible that the CO' topology
  is strictly coarser than the UC~one. For example, denote the
  Hawaiian earring by~\(X\), and equip its fundamental groupoid with
  the CO' topology. Then Fabel~\cite{Fabel2011HE-Counterexample} shows
  that the multiplication on the fundamental group of~\(X\) is
  discontinuous in the subspace topology. This fact, in the light of
  Lemma~\ref{lem:CO-coarser} and
  Corollary~\ref{cor:subspace-top-on:fibre-UC-top} implies that the
  CO' topology on~\(\FGd(X)\) is strictly coarser than the
  UC~topology. In fact,
  Theorem~1.1~\cite{Calcut-McCarthy2009Top-Fund-Gp} proves a stronger
  fact, namely, the restrictions of these topologies on the
  fundamental group of a locally path connected space agree \emph{iff}
  the space is semilocally simply connected. Thus this theorem also
  explains the issue with the topology on the fundamental group of the
  Hawaiian earring.
\end{remark}

\begin{remark}
  \label{rem:Top-FGd-stroger-than-Alg-FGp}
  Let \(X\) be the infinite dimensional separable Hilbert space of
  square summable complex sequences. Let~\(X^*\) be its dual
  space. Then \(X^*\) carries two \emph{distinct} topologies, namely,
  the norm topology and weak\(^*\)\nb-topology. With respect to both
  topologies, \(X^*\) is a convex topological vector space.  Therefore
  \(X^*\) is path connected, locally path connected and simply
  connected in both the topologies. The algebraic fundamental groupoid
  of \(X^*\) with norm topology is isomorphic to the algebraic
  fundamental groupoid of \(X^*\) with weak\(^*\)\nb-topology---in
  both cases this is isomorphic to the groupoid of trivial equivalence
  on~\(X^*\) as in
  Proposition~\ref{prop:gpd-of-simply-conn-space}. But their
  topological fundamental groupoids are different as the \emph{spaces
    of the units} of these groupoids are different. This observation
  ensure that the topological fundamental groupoid is a sharper
  invariant than the algebraic one.
\end{remark}

\subsection{Functoriality}
\label{sec:functoriality}

\begin{lemma}
  \label{lem:cont-gpd-homo}
  Let \(f\colon X \to Y\) be a continuous map between two locally path
  connected and semilocally simply connected topological spaces. Then
  \(f_*\colon \FGd(X) \to \FGd(Y)\) given by
  \([\alpha] \mapsto f_*([\alpha]) \defeq [f\circ \alpha]\) is a
  continuous groupoid homomorphism.
\end{lemma}
\begin{proof}
  It is a standard fact that \(f\mapsto f_*\) is a functor from the
  category of topological spaces to (algebraic) fundamental
  groupoids. Therefore, all we need to show the continuity of
  continuity of \(f_*\). For that let \(N([f\circ \alpha], V, U)\) be
  an UC\nb-open neighbourhood of \([f \circ \alpha] \in \FGd(Y)\),
  where \(V\) and \(U\) be two open neighbourhoods of \(f(\alpha(1))\)
  and \(f(\alpha(0))\), respectively. By continuity of \(f\) there
  exist open neighbourhoods \(V^\prime\) of \(\alpha(1)\) and
  \(U^\prime\) of \(\alpha(0)\) satisfying \(f(V^\prime) \subseteq V\)
  and \(f(U^\prime) \subseteq U\). Now it can be readily seen that
  \(f_*(N([\alpha], U', U')) \subseteq N([f\circ \alpha], U,
  V)\). \qedhere
\end{proof}

Lemma~\ref{lem:cont-gpd-homo} can be proved using the CO' topology as
well. For that, one should notice that the \(f\colon X\to Y\) induces
a path homotopy preserving continuous map~\(\PathS{X}\to \PathS{X}\)
which, in turn, induces a map of fundamental groupoids.

Let \(\TCat\) be the category of path connected, locally path
connected and semilocally simply connected topological spaces;
\(\GCatAlg\) the category of algebraic groupoids with algebraic
homomorphisms are morphisms; and \(\GCat\) be obtained by
enriching~\(\GCatAlg\) with topology. Thus~\(\GCat\) is the category
consist of topological groupoids with continuous groupoid homomorphism
as morphism. It is well-known that sending a space to its fundamental
groupoid is a functor \(\TCat\to
\GCatAlg\). Lemma~\ref{lem:cont-gpd-homo} combined with this standard
fact gives us the next result:

\begin{theorem}
  The assignment \(\FGd\colon \TCat\to \GCat\) that sends a an space
  \(X\mapsto \FGd(X)\) and a map of space \(X\xrightarrow{f}Y\) to
  \(\FGd(X)\xrightarrow{f_*} \FGd(Y)\) in
  Lemma~\ref{lem:cont-gpd-homo} is a covariant functor.
\end{theorem}

\subsection{The fundamental groupoid of a group}
\label{sec:fund-group-group}

Let \(p\colon H\to G\) be a group homomorphism. Then~\(H\) acts on the
left of~\(G\) through~\(p\) as \(\eta \cdot \gamma = p(\eta) \gamma \)
for \(\gamma \in G\) and \(\eta\in H\). We denote the corresponding
transformation groupoid by~\(H\ltimes_p G\).

For a path connected group~\(G\), we follow the standard custom of
considering \(\pi_1(G)\) as~\(\pi_1(G,1_G)\) where \(1_G\in G\) is the
identity; this fundamental group is abelian~\cite[Corollary
3.21]{Rotman1988Alg-Top-Book}. Recall that for a path connected,
locally path connected and semilocally simply connected group~\(G\),
its simply connected cover~\(H\) is also a path connected, locally
path connected and semilocally simply connected \emph{group} and the
covering map~\(p\colon H\to G\) is a group homomorphism. In this
section, we fix to \(G,H\) and~\(p\) as in the last sentence. Our goal
is to prove that the fundamental groupoid of~\(G\) is isomorphic to
the transformation groupoid \(H\ltimes_p G\).

Throughout this section, we shall consider the universal cover~\(H\)
as the fibre~\(\FGd(G)_{1_G}\). In this case, the covering map
\(p\colon H\to G\) is given by \(p([\gamma]) =\gamma(1)\) for
\([\gamma]\in H\).  And for \([\eta], [\gamma] \in H\), the group
multiplication~\(\otimes\) in~\(H\) is given by
\([\gamma]\otimes [\eta] =[\gamma\otimes' \eta]\) where
\(\gamma\otimes' \eta\) is the path \(t\mapsto \gamma(t) \eta(t)\). We
shall abuse the notation and write~\(\otimes\) instead of~\(\otimes'\)
now on.

For any path \(\gamma\) in \(G\) and any \(g\in G\), define the
translated path~\(\gamma\cdot g\) as \(t\mapsto \gamma(t)g\) where
\(t\in \UInt\). Note that
\[
  (\eta\otimes \gamma)\cdot g = (\eta\otimes \gamma \cdot g)
\]
where \(\eta,\gamma\in \PathS{G}_{1_G}\) and~\(g\in
G\). Lemma~\ref{lem:joint-of-translated-path} explains how the
concatenation and translation of paths behave when
combined. Occasionally, we shall use~\(\cdot\) also to denote the
multiplication in a group---this shall lead to no confusion.  In
general, the concatenation, product of paths and the translates of
paths have following properties: for paths
\(\gamma,\eta,\delta,\theta\in \PathS{G}\) with
\(\delta(0) = \theta(1)\), and \(g\in G\),
\begin{align*}
  (\eta\otimes \gamma) \cdot g &=   \eta\otimes(\gamma \cdot g ),\\
  (\delta\oblong \theta) \cdot g & = (\delta\cdot g)\oblong (\theta
                                   \cdot g).
\end{align*}

\noindent The multiplication by~\(g\) above is a homeomorphism
of~\(\PathS{G}\) which respected the (endpoint preserving) path
homotopy, that is, \( [\eta \cdot g] = [\eta] \cdot a\) for
the~\(\eta\) above. Therefore, for the paths as above we also have
\begin{align}
  [\eta\otimes \gamma] \cdot g &=   [\eta]\otimes
                                 [\gamma \cdot g] \label{eq:product-with-translate-1},\\
  [\delta\oblong \theta] \cdot g & = [\delta\cdot g]\oblong [\theta
                                   \cdot g]. \label{eq:product-with-translate-2}
\end{align}

Finally, we realise the space~\(\FGd(G)\) as the fibre product
\(\FGd(G)\times_{s,\base[\FGd(G)], \Id_{G}} G\defeq \{([\gamma], g) :
g\in G \text{ and } \gamma(0)=g\}\); this identification is useful to
prove our main result of this section.  Note that the basic open set
in~\(\FGd(G)\times_{s,\base[\FGd(G)], \Id_{G}} G\) corresponding to
the basic open set~\(N([\alpha], U,V)\subseteq \FGd(G)\)
is~\(N([\alpha], U,V) \times_{s,\base[\FGd(G)], \Id_{G}} V\).

\begin{lemma}
  \label{lem:cont-G-act-on-cover}
  The mapping~\(J\colon H\times G \to \FGd(G)\) given
  by~\(([\gamma], g)\mapsto ([\gamma\cdot g], g)\) is a homeomorphism.
\end{lemma}
\begin{proof}
  Let \(N([\alpha], U,V)\times_{s,\base[\FGd(G)], \Id_{G}}V\) be a
  basic open set in~\(\FGd(G)\) wherein~\(U,V\subseteq G\) are path
  connected and semilocally simply connected open sets. Its inverse
  image under~\(J\) is the set~\(N([\alpha],U)\times V\) which is a
  basic open set in the product topology. Conversely, if
  \((N([\alpha],U)\times V \subseteq H\times G\) is a basic open set
  wherein~\(U\) and~\(V\) are path connected and semilocally simply
  connected open sets in~\(G\) (such sets form a basis of~\(G\)), then
  \(J(N([\alpha],U)\times V) = N([\alpha], U, V)
  \times_{s,\base[\FGd(G)], \Id_{G}} V\). Therefore,~\(J\) is
  open. One can easily check that~\(J\) is invertible with its inverse
  given by~\(J\inverse([\gamma], g) = ([\gamma\cdot g\inverse], g)\).
\end{proof}
  
\noindent Reader may try to prove the last lemma using the CO'
topology. In that case, one needs to carefully justify that the
following two topologies on~\(\FGd(G)_{1_G}\times G\) agree:
\begin{itemize}
\item the product topology on~\(\FGd(G)_{1_G}\times G\),
\item the quotient topology on it for the obvious quotient map
  \(\PathS{G}_{1_G}\times G \to \FGd(G)_{1_G}\times G\).
\end{itemize}

Let \(\eta,\gamma\) be paths in~\(G\) starting at the
identity~\(1_G\). Then the paths \(\eta\cdot \gamma(1)\) is
concatinable with \(\gamma\) as
\((\eta\cdot\gamma(1))(0) = \eta(0)\gamma(1) = \gamma(1)\). Therefore,
the statement of the next lemma makes sense:
  
\begin{lemma}\label{lem:joint-of-translated-path}
  Let \(\eta,\gamma\) be paths in~\(G\) starting at the
  identity~\(1_G\).  Then, for any \(g\in G\),
  \[
    [(\eta\cdot \gamma(1) \oblong \gamma)\cdot g] = [(\eta\otimes
    \gamma)\cdot g].
  \]
\end{lemma}
\begin{proof}
  Recall that for two continuous functions, corresponding minimum and
  maximum functions are continuous.  Then for \(0\leq t\leq 1\),
  \begin{multline*}
    \bigg(\big(\eta \cdot \gamma(1) \oblong \gamma\big)\cdot g\bigg)
    (t) = \bigg( \big(\eta\cdot \gamma(1) \oblong \gamma \big) (t)
    \bigg)\cdot g =\begin{cases}
      \gamma(2t)g, & \quad \textup{ if } 0 \leq t \leq \frac{1}{2},\\
      \eta(2t-1)\gamma(1)g, & \quad \textup{ if } \frac{1}{2} \leq t
      \leq 1
    \end{cases}\\
    = \eta\big(\max\{2t-1, 0\} \big)\cdot \gamma\big(\min\{2t,
    1\}\big)\cdot g.
  \end{multline*}
  For \(s,t\in [0,1]\), consider the function
  \[
    H(s,t) = \eta\big(\max\{t+(t-1)s, 0\}\big)\cdot
    \gamma\big(\min\{t+ ts, 1\}\big) \cdot g.
  \]
  Then \(H\) implements a path homotopy between
  \(H(0,*) = (\eta\otimes \gamma)\cdot g\) and
  \(H(1,*)= (\eta\cdot \gamma(1) \oblong \gamma)\cdot g \). This is
  because~\(H\) is continuous, and
  \[
    H(s,0) = g =\eta(0) \gamma(0) g = \big((\eta\cdot \gamma(1)
    \oblong \gamma)\cdot g\big) (0),\] and
  \[
    H(s,1) = \eta(1)\gamma(1) g = ((\eta\otimes \gamma) \cdot g)(1) =
    \left((\eta \gamma(1) \oblong \gamma) g\right)(1)
  \]
  for \(0\leq s\leq 1\).
\end{proof}

\noindent Note that, if \(g=1_G\), then
Lemma~\ref{lem:joint-of-translated-path} allows us to express the
product in the universal covering group ~\(H\) in terms of the path
homotopies. We need this observation, to prove
Theorem~\ref{thm:FGd-of-gp:as-tranf} next. Before we do that,
following are three observations about
Lemma~\ref{lem:joint-of-translated-path} which we find worth listing;
however, they shall not be used in this article.

(I) Suppose that \(\gamma\) is a path in~\(G\) starting at the
identity and~\(\eta\) a path starting at~\(\gamma(1)\). Then
\[
  [\eta \oblong \gamma] = [\eta \otimes \gamma]
\]
by replacing \(\eta\) by~\(\eta\cdot\gamma(1)\inverse\) and
taking~\(g=\gamma(0)=1_G\) in the lemma.

(II) Next, let~\(\eta\) and~\(\gamma\) be paths in~\(G\) with
\(\eta(0) = \gamma(1)\).  Applying the lemma on the
paths~\(\eta \cdot \gamma(1)\inverse, \gamma\cdot \gamma(0)\inverse\)
starting at~\(1_G\) and the group element~\(g = \gamma(0)\), we get
\[
  [\eta \oblong \gamma \cdot \gamma(0)\inverse ] \gamma(0)= [\eta
  \otimes \gamma].
\]

(III) Finally, in the general setting, let \(\gamma\) and~\(\eta\) be
paths in~\(G\).  Applying the last lemma on the paths
\(\eta\cdot \eta(0)\inverse, \gamma\cdot\gamma(0)\inverse\)
and~\(g = \gamma(0)\) gives us
\[ \left([\eta\cdot \eta(0)\inverse(\gamma(1)
    \gamma(0)\inverse)\right) \oblong \left(\gamma\cdot
    \gamma(0)\inverse]\right) \cdot \gamma(0) = [\left(\left(\eta\cdot
      \eta(0)\inverse\right) \otimes \left(\gamma
      \cdot\gamma(0)\inverse\right)\right) \cdot \gamma(0)].
\]
Applying Equation~\eqref{eq:product-with-translate-1} on the last term
implies:
\[ [\left(\eta\cdot \eta(0)\inverse \left(\gamma(1) \gamma(0)\inverse
    \right) \right)\oblong \left(\gamma\cdot \gamma(0)\inverse\right)]
  \cdot \gamma(0) = [(\eta\cdot \eta(0)\inverse ) \otimes \gamma].
\]

\begin{theorem}
  \label{thm:FGd-of-gp:as-tranf}
  Let \(G\) be a locally path connected, path connected, semilocally
  simply connected topological group, and \(p\colon H\to G\) be its
  universal covering space. Then the topological fundamental
  groupoid~\(\FGd(G)\) is isomorphic to the transformation
  groupoid~\(H\ltimes_p G\).
\end{theorem}
\begin{proof}
  Define
  \[
    J \colon H\ltimes_p G \to \FGd(G) \quad \text{by}\quad J(([
    \gamma], g))= ([\gamma\cdot g], g)
  \]
  where \(([\gamma], g)\in H\times
  G\). Lemma~\ref{lem:cont-G-act-on-cover} shows that~\(J\) is a
  homeomorphism. All we need to do is to prove that~\(J\) is an
  (algebraic) homomorphism of groupoids.  For that, take two
  composable elements
  \(([ \eta], h), ([ \gamma], g)\in H\ltimes_p G\). Then
  \begin{itemize}
  \item \(h= p([\gamma]) g \defeq \gamma(1)g \), and
  \item
    \(([ \eta], h) ([ \gamma], g) = ([\eta]\otimes [\gamma], g) =
    ([\eta \otimes \gamma], g)\).
  \end{itemize}
  As a consequence,
  \begin{equation}
    \label{eq:image-of-prod-under-J}
    J(([ \eta], h) ([ \gamma], g)) = J([\eta \otimes \gamma], g) =
    ([(\eta \otimes \gamma) \cdot g], g)
  \end{equation}
  On the other hand, the images of these elements
  \(J( [\eta], h) = ([\eta\cdot h], h)\) and
  \(J([ \gamma], g) = ([\gamma\cdot g], g)\) are clearly composable in
  \(\FGd(G)\) as
  \((\gamma \cdot g)(1)= \gamma(1) g = h = (\eta\cdot
  h)(1)\). Furthermore, the concatenation of these images
  in~\(\FGd(G)\) is
  \[
    J([\eta], h) J([ \gamma], g) = ([\eta\cdot h], h) ([\gamma\cdot
    g], g) = ([(\eta\cdot h)\oblong (\gamma\cdot g)], g).
  \]
  By using Property~\eqref{eq:product-with-translate-1}, take
  the~\(g\) out in the path
  homotopy~\([(\eta\cdot h)\oblong (\gamma\cdot g)]\) above; then the
  last term in above computation equals
  \[
    ([(\eta\cdot h g\inverse)\oblong \gamma] \cdot g, g).
  \]
  Now apply Lemma~\ref{lem:joint-of-translated-path} to the first
  coordinate above so the last term equals
  \[
    ([(\eta \otimes \gamma)\cdot g], g) = J(([ \eta], h) ([ \gamma]
    g))
  \]
  as we wanted, cf. Equation~\eqref{eq:image-of-prod-under-J} above.
\end{proof}

Note that last theorem also shows that~\(J\) is simply an isomorphism
of algebraic groupoids if the topologies are ignored.

Since the groupoids \(\FGd(G)\) and~\(H\ltimes_p G\) in
Theorem~\ref{thm:FGd-of-gp:as-tranf} are isomorphic, the isotropies of
corresponding points are also isomorphic. The isotropy of the
unit~\(1_G\in G\) in the transformation groupoid~\(H\ltimes_p G\) is
\(\{\eta\in H : p(\eta) = e\} = \mathrm{ker}(p)\). Therefore,
\(\FGp(G, e)\) is isomorphic to the kernel of~\(p\).

An ending remark for this section is following:
Theorem~\ref{thm:FGd-of-gp:as-tranf} shows that for the fundamental
groupoid not to have the form of a transformation groupoid it is
necessary that the underlying space is not a topological
group. Following are the path connected, locally path connected and
semilocally simply connected spaces which are known to carry no group
structure (compatible with their given topologies).

 \begin{example}
   The sphere \(\Sph^n\subseteq \R^{n+1}\) for \(n\neq 0,1,3\).
 \end{example}

 \begin{example}\label{exa:fig-eight}
   The figure of eight in the plane. The fundamental group of a
   topological group is abelian~\cite[Exercise 7(d), \S
   52]{Munkress1975Topology-book}. The fundamental group of the figure
   of eight is not abelian.
 \end{example}

 \begin{example}
   The wedge sum of \(n\)-circles \(\vee_n\, \Sph^1\) for \(n>1\); for
   same reason as Example~\ref{exa:fig-eight}.
 \end{example}

 \section{The point-set topology\except{toc}{\protect\footnote{Thanks
       to Massey~\cite[Chapter five,
       Section~12]{Massey1977Alg-Top-Intro-reprint} for this title!}}
   of the fundamental groupoid}
 \label{sec:descr-fund-group}

 In the second subsection, we discuss the point-set topology of the
 fundamental groupoid~\(\FGd(X)\). We relate this topology with that
 of~\(X\) and the simply connected covering space of~\(X\). For this
 purpose, in the next subsection, we need an alternate construction of
 the fundamental groupoid which derived from a standard construction
 involving equivalence of groupoids.

 \subsection{One more description of the fundamental groupoid}
 \label{sec:one-more-description}

 \begin{definition}[Equivalence of
   groupoids~{\cite{Muhly-Renault-Williams1987Gpd-equivalence}}]
   \label{def:equi}
   Let \(G\) and \(H\) be groupoids. We call a space \(X\) an
   \(G\)-\(H\)-equivalence if \(X\) is a left \(G\)- and right
   \(H\)\nb-space, and following properties are satisfied:
   \begin{enumerate}
   \item the actions commute, that is,
     \(\gamma(x\eta) = (\gamma x)\eta\) for all appropriate
     \(\gamma\in G, x\in X\) and~\(\eta\in H\);
   \item both the actions are free and proper;
   \item the momentum maps induce homeomorphism \(G\7 X \to \base[H]\)
     and \(X/H \to \base[G]\).
   \end{enumerate}
 \end{definition}

 The next is a well-known example of equivalence of groupoids. It
 appeared in~\cite{Muhly-Renault-Williams1987Gpd-equivalence} as a
 basic example of equivalence of groupoids.
  
 \begin{example}[{\cite[Proposition~2.41]{Williams2019A-Toolkit-Gpd-algebra}}]
   \label{exa:equi-gpd}
   Let~\(H\) be a group and~\(X\) a right \(H\)\nb-space. Assume that
   the action of~\(H\) is free and proper. Then
   \(G\defeq (X\times X)/ H\) is a topological groupoid; here the
   action of~\(H\) on the product space is given
   by~\((x,y)\eta = (x\eta,y\eta)\) where \(x,y\in X\)
   and~\(\eta\in H\). The groupoid~\(G\) acts freely and properly
   on~\(X\) from the left. Moreover, this way~\(X\) implements an
   equivalence between~\(G\) and~\(H\).
 \end{example}

 We briefly describe the groupoid structure on \((X\times X)/ H\) and
 its action on~\(X\) mentioned in the last example. Denote the
 equivalence class of \((x,y)\in X\times X\) by \([x,y]\); then define
 the range and source map from \((X\times X) /H \to X/H\) by
 \(r([x,y]) = x H\) and \(s([x,y]) = y H\) respectively; here \(xH\)
 is the \(H\)\nb-orbit of~\(x\). Then \([x,y]\) and \([z,w]\) are
 composable \emph{iff} \(yH = zH\), that is, \(z= yh\) for some
 \(h\in H\). Now the composite of these two arrows is given by
 \([x,y] [z,w] = [x,y] [y\cdot h, w] = [x, w\cdot h^{-1}]\); the
 composite is well-defined due to freeness of the action (see
 \cite[Lemma~2.39]{Williams2019A-Toolkit-Gpd-algebra} for
 details). Finally, the inverse \([x,y]^{-1} = [y,x]\). The unit space
 of the groupoid \(G\defeq (X\times X)/ H\) is identified with
 \(X/H\).

 The momentum map for the action of \((X\times X)/H\) is the quotient
 map \(q\colon X \to X/H\). The action is given by \([x,y] y = x\)
 where \(([x,y],y) \in (X\times X)/H \times_{s, X/H, q} X\);
 see~\cite[Proposition~2.41]{Williams2019A-Toolkit-Gpd-algebra}
 details.

 Recall from Example~\ref{exa:fund-gp-action} that for a Hausdorff,
 path connected, locally path connected and semilocally simply
 connected space~\(X\), the deck transformation action of~\(\FGp(X)\)
 on~\(\tilde{X}\) is free and proper. Therefore, using
 Example~\ref{exa:equi-gpd} we notice that the quotient
 space~\((\tilde{X}\times \tilde{X})/ \FGp(X)\) is a groupoid
 equivalent to~\(\FGp(X)\) where the equivalence is implemented by the
 simply connected covering space~\(\tilde{X}\). But Example~2.3
 in~\cite{Muhly-Renault-Williams1987Gpd-equivalence} already shows
 that~\(\FGp(X)\) is equivalent to~\(\FGd(X)\) via the simply
 connected covering space. Then~\cite[Lemma
 2.38]{Williams2019A-Toolkit-Gpd-algebra} implies
 that~\((\tilde{X}\times \tilde{X})/ \FGp(X) \iso \FGd(X)\) as
 topological groupoids. This result is summarised as follows:

 \begin{proposition}
   \label{prop:quotient-description-of-fund-gp}
   Let \(X\) be Hausdorff, path connected, locally path connected and
   semilocally simply connected space; and let \(\tilde{X}\) be its
   simply connected cover. Let \(G\) be the topological
   groupoid~\((\tilde{X} \times \tilde{X})/ \FGp(X)\) associated with
   the free an proper action of \(\FGp(X)\) on~\(\tilde{X}\) as in
   Example~\ref{exa:equi-gpd}. Then~\(G\) is isomorphic to the
   topological fundamental groupoid of~\(X\).
 \end{proposition}

 \subsection[The point-set topology]{The point-set topology}
 \label{sec:prop-topol-fg}

 Massey~\cite[Ch.5,\S12]{Massey1977Alg-Top-Intro-reprint} discusses
 the point-set topology on the covering space in connection with that
 of the base space. He shows or gives it as an exercise to prove that
 if a locally path connected, semilocally simply connected space has
 any of the following properties, then its simply connected cover too
 has the property: second countability, Hausdorffness, regularity,
 complete regularity, local compactness, being a separable metric
 space. Along the similar lines, we study the relation between the
 point-set topologies of~\(X,\tilde{X}, \PathS{X}\) and the
 fundamental groupoid~\(\FGd(X)\). Our main goal is to investigate
 when the groupoid is Hausdorff or locally compact or second
 countable. Before we start proving the main propositions, here is a
 lemma:

 \begin{lemma}
   \label{lem:closed-equi-rel}
   Assume \(X\) is a topological space with an equivalence
   relation~\(\sim\) having open quotient map~\(q\colon X\to
   X/\sim\). Then \(X/\sim\) is Hausdorff iff
   \(\sim\subseteq X\times X\) is a closed subset.
 \end{lemma}
 \begin{proof}
   Write \(Y\) for \(X/\sim\). Let
   \(q\times q: X\times X\to Y \times Y\) be the continuous mapping
   \((x,y)\mapsto (q(x), q(y))\) where \(x,y\in X\).  If the diagonal
   \(\mathrm{dia}(Y)\subseteq Y \times Y\) is closed, then
   \(\sim = (q\times q)\inverse(\mathrm{dia}(Y))\) is an obviously
   closed subspace of~\(X\times X\).

   For the converse, note that~\(q\times q\) is open map since~\(q\)
   is open. Now assume that \(\sim\) is a closed subspace
   of~\(X\times X\). Then note that
   \[
     (q\times q)((X\times X)- \sim) = (Y\times Y)- \mathrm{dia}(Y).
   \]
   The left hand side, and hence the right side, of above equation is
   an open set. Thus the diagonal
   \(\mathrm{dia}(Y)\subseteq Y\times Y\) is closed.
 \end{proof}

 \begin{proposition}
   \label{prop:Haus-FGd}
   Let \(X\) be a locally path connected and semilocally simply
   connected space. Let \(\sim\) be the equivalence relation of path
   homotopy on~\(\PathS{X}\). Consider the following assertions.
   \begin{enumerate}
   \item\label{it:X-Hd} \(X\) is Hausdorff.
   \item\label{it:X-tilde-Hd} \(\tilde{X}\) is Hausdorff.
   \item\label{it:PX-Hd} \(\PathS{X}\) is Hausdorff.
   \item\label{it:GdX-Hd} \(\FGd(X)\) is Hausdorff.
   \item\label{it:rel-Hd} \(\sim\subseteq \PathS{X}\times \PathS{X}\)
     is a closed subspace.
   \end{enumerate}
   Then~(\ref{it:X-Hd}),~(\ref{it:PX-Hd}),~(\ref{it:GdX-Hd})
   and~(\ref{it:rel-Hd}) are equivalent, and
   (\ref{it:X-Hd})\(\implies\)(\ref{it:X-tilde-Hd}).
 \end{proposition}

 \begin{proof}
   \noindent (\ref{it:X-Hd})\(\iff\)(\ref{it:PX-Hd}) follows from
   Property~\ref{prop:CO-Hausdorff}.

   \noindent (\ref{it:GdX-Hd})\(\iff\)(\ref{it:rel-Hd}): Since the
   quotient map \(\PathS{X}\to \FGd(X)\) is open, this follows
   from~Lemma~\ref{lem:closed-equi-rel}.

   \noindent (\ref{it:GdX-Hd})\(\implies\) (\ref{it:X-Hd}), as
   \(X, \tilde{X}\subseteq \FGd(X)\) are
   subspaces~Corollary~\ref{cor:subspace-top-on:fibre-UC-top}.

   \noindent (\ref{it:X-Hd})\(\implies\)(\ref{it:GdX-Hd}): Suppose
   \(X\) is Hausdorff space and two distinct points
   \([\alpha], [\beta]\in \FGd(X)\) are given. Then there are two
   cases: one of the endpoints of~\(\alpha\) and~\(\beta\) differ, or
   both these paths have same initial points and same terminal
   points. In the first case, first assume
   that~\(\alpha (0) \neq \beta(0)\). Choose a neighbourhoods
   \(\alpha(0)\in W\) and~\(\beta(0)\in W'\) in~\(X\) with
   \(W\cap W'= \emptyset\). And choose any neighbourhoods
   \(\alpha(1)\in V\) and~\(\beta(1)\in V'\).  Then clearly
   \(N([\alpha],V, W) \cap N([\beta], V', W') =\emptyset\).  A similar
   argument works if~\(\alpha (1) \neq \beta(1)\).

   Consider the second case, namely, \(\alpha(0) =\beta(0)\) and
   \(\alpha(1) =\beta(1)\). Since \(X\) is semilocally simply
   connected choose relatively inessential neighbourhoods \(U\) of
   \(\alpha(0)\) and \(V\) of \(\alpha(1)\). Then
   \(N([\alpha],V,U) \cap N([\beta], V,U) = \emptyset\). If this
   intersection is not empty, we get a contradiction as follows:
   assume that \([\gamma] \in N([\alpha],V,U) \cap N([\beta],
   V,U)\). Then
   \([\gamma] = [\eta \oblong \alpha \oblong \theta] = [\eta' \oblong
   \beta \oblong \theta']\) for some \(\eta, \eta' \in \PathS{V}\) and
   \(\theta, \theta'\in \PathS{U}\) with
   \[
     \gamma(0) =\theta(0) =\theta'(0),\quad \gamma(1)
     =\eta(1)=\eta'(1).
   \]
   Moreover, since~\(\alpha\) and~\(\beta\) have the same endpoints,
   we also have that \(\theta(1) =\theta'(1)\) and
   \(\eta(0) =\eta'(0)\). Therefore,
   \begin{equation}\label{equ:Haus-FGd}
     [\alpha] = [\eta^{-1} \oblong \eta' \oblong \beta \oblong \theta' \oblong \theta^{-1}] = [\eta^{-1} \oblong \eta'] [\beta] [\theta' \oblong \theta^{-1}] =[\beta].
   \end{equation}
   The last equality holds because \(U\) and \(V\) are relatively
   inessential neighbourhoods. Equation~\eqref{equ:Haus-FGd}
   contradicts that~\([\alpha] \) and \([\beta]\) are distinct.

   \noindent Finally, for the proof of
   (\ref{it:X-Hd})\(\implies\)(\ref{it:X-tilde-Hd}),
   consider~\(\tilde{X}\) as a subspace of~\(\FGd(X)\). Then the same
   proof as in the first case above works.
 \end{proof}

 Recall from Section~\ref{sec:groupoids-prel} that for us locally
 compact spaces are not necessarily Hausdorff.

 \begin{proposition}
   \label{prop:local-cpt}
   Consider the following statements about a path connected, locally
   path connected and semilocally simply connected space~\(X\).
   \begin{enumerate}
   \item\label{it:X-lc} \(X\) is locally compact.
   \item\label{it:X-tilde-lc} \(\tilde{X}\) is locally compact.
   \item\label{it:FGX-lc} \(\FGd(X)\) is locally compact.
   \end{enumerate}
   Then,~(\ref{it:X-lc})\(\iff\)(\ref{it:X-tilde-lc}),
   and~(\ref{it:FGX-lc})\(\implies\)(\ref{it:X-tilde-lc}),(\ref{it:X-lc}).
   If~\(X\) is also Hausdorff, then all are equivalent.
 \end{proposition}

 \noindent As discussed after Property~\ref{prop:CO-paracpt}, local
 compactness of~\(\PathS{X}\) is not a practical expectation.

 \begin{proof}[Proof of Proposition~\ref{prop:local-cpt}]
   \noindent (\ref{it:X-lc})\(\iff\)(\ref{it:X-tilde-lc}): This
   follows because~\(X\) and \(\tilde{X}\) are locally homeomorphic
   via the covering map. We sketch the proof for the sake of clarity:
   suppose~\(p\colon \tilde{X} \to X\) is the covering map. Assume
   that~\(\tilde{X}\) is locally compact, and~\(x\in X\) is
   given. Choose a point~\(\tilde{x}\) in the preimage of~\(x\). Let
   \(U\subseteq \tilde{X}\) be a neighbourhood such that
   \(p(U)\subseteq X\) is open and~\(p\colon U\to p(U)\) is a
   homeomorphism. Get a neighbourhood~\(V\) of~\(\tilde{x}\) such that
   \(\overline{V}\) is compact and \(\subseteq U\) using
   Lemma~\ref{lem:loc-cpt}. Then \(p(\overline{V})\) is a compact
   neighbourhood of~\(x\). The converse is proved similarly.
  
   \noindent (\ref{it:FGX-lc})\(\implies\)(\ref{it:X-tilde-lc}): A
   locally compact space in our sense is~\(T_1\). Therefore, the
   covering space \(\tilde{X} \subseteq \FGd(X)\) is a closed subspace
   (remark after
   Corollary~\ref{cor:subspace-top-on:fibre-UC-top}). Being a closed
   subspace of a locally compact space
   (\cite[Proposition~1.2]{Tu2004NonHausdorff-gpd-proper-actions-and-K}),
   \(\tilde{X}\) is closed.

   Since~(\ref{it:X-lc}) and~(\ref{it:X-tilde-lc}) are proved to be
   equivalent, (\ref{it:FGX-lc})\(\implies\)(\ref{it:X-lc}).

   Finally, assume that \(X\) Hausdorff. Then~\(\tilde{X}\) is
   Hausdorff due to Proposition~\ref{prop:Haus-FGd}. Therefore,
   \(\tilde{X}\times \tilde{X}\) is a locally compact Hausdorff
   space. Now Proposition~\ref{prop:quotient-description-of-fund-gp}
   implies that~\(\FGd(X)\) is locally compact Hausdorff being
   quotient space of a proper action. This proves
   (\ref{it:X-lc})\(\implies\)(\ref{it:FGX-lc}).
 \end{proof}

 Recall from our convention, Section~\ref{sec:groupoids-prel}, that
 paracompact and second countable spaces are Hausdorff.

 \begin{proposition}
   \label{prop:sec-cble}
   Let \(X\) be locally path connected and semilocally simply
   connected space. Then following are equivalent:
   \begin{enumerate}
   \item\label{it:X-sc} \(X\) is second countable.
   \item\label{it:X-tilde-sc} \(\tilde{X}\) is second countable.
   \item\label{it:PX-sc} \(\PathS{X}\) is second countable.
   \item\label{it:FGX-sc} \(\FGd(X)\) is second countable.
   \end{enumerate}
 \end{proposition}
 \begin{proof}
   \noindent (\ref{it:PX-sc})\(\implies\)(\ref{it:FGX-sc}): Assume
   \(\PathS{X}\) is second countable. Then the fundamental groupoid is
   second countable as it is image of~\(\PathS{X}\) under the quotient
   map which is open due to
   Corollary~\ref{cor:quotient-from-PX-to-PiX-open}.

   \noindent (\ref{it:FGX-sc})\(\implies\)(\ref{it:X-sc}) or
   (\ref{it:X-tilde-sc}): Since \(X\) and~\(\tilde{X}\) are subspaces
   of~\(\FGd(X)\), they are second countable~\cite[Theorem
   30.2]{Munkress1975Topology-book}.
  
   \noindent (\ref{it:X-sc})\(\iff\)(\ref{it:PX-sc}): Due to
   Property~\ref{prop:CO-sec-ctble}.

  \noindent 
  (\ref{it:X-tilde-sc})\(\implies\)(\ref{it:FGX-sc}): This is a
  consequence of
  Proposition~\ref{prop:quotient-description-of-fund-gp}.
\end{proof}

The proof of the last proposition, except
(\ref{it:X-tilde-sc})\(\implies\)(\ref{it:FGX-sc}), holds even if
second countable spaces are not assumed to be Hausdorff.

Next discussion has some loose ends. Before we proceed, note that the
space of units of a Hausdorff groupoid~\(G\) is a closed subspace
of~\(G\): consider the map \(G\to G\times G\) that sends
\(\gamma\in G\) to \((\gamma, \gamma\inverse)\). The diagonal
in~\(G\times G\) is closed. And~\(\base\) is inverse image of the
diagonal under the above continuous map.

\begin{proposition}
  \label{prop:paracpt}
  Let \(X\) be path connected, locally path connected and semilocally
  simply connected space. Consider the following statements:
  \begin{enumerate}
  \item\label{it:X-pc} \(X\) is paracompact.
  \item\label{it:X-tilde-pc} \(\tilde{X}\) is paracompact.
  \item\label{it:FGX-pc} \(\FGd(X)\) is paracompact.
  \end{enumerate}
  Then
  (\ref{it:FGX-pc})\(\implies\)(\ref{it:X-tilde-pc}),(\ref{it:X-pc}).
  Additionally, if we assume that \(\FGp(X)<\infty\), then
  (\ref{it:X-pc})\(\implies\)(\ref{it:X-tilde-pc}).
\end{proposition}
\begin{proof}
  (\ref{it:FGX-pc})\(\implies\)(\ref{it:X-tilde-pc}),(\ref{it:X-pc}):
  Follows because \(X, \tilde{X}\subseteq \FGd(X)\) are closed
  subspaces---\(X\) is closed as it is the space of units, and
  \(\tilde{X}\) is closed for being the inverse image of a point under
  the range map.
  
  \noindent (\ref{it:X-pc})\(\implies\)(\ref{it:X-tilde-pc}): In this
  case the covering map \(p\colon \tilde{X} \to X\) is a perfect
  map. Therefore, \(\tilde{X}\) is paracompact if \(X\) is so,
  see~\cite[Exercise 8, page 260]{Munkress1975Topology-book}.
\end{proof}

As observed in Property~\ref{prop:CO-paracpt}, the path space of a
metric space is paracompact.

Zabrodsky discuses metrisability of~\(\tilde{X}\)
in~\cite{Zabrodsky1964Covering-spaces-of-paracompact-spaces}; he
proves that if~\(X\) is metrisable, then so is~\(\tilde{X}\) in a
\emph{nice} way; depending on if~\(X\) is locally path connected,
local behaviour of these metrics change, see Theorem~2 vs Theorem~3
in~\cite{Zabrodsky1964Covering-spaces-of-paracompact-spaces}.

\medskip

\paragraph{\itshape Acknowledgement:} We are grateful for the
discussions with Prahlad Vaidyanathan and Angshuman Bhattacharya which
lead to this article. We thank Ralf Meyer, Jean Renault and Atreyee
Bhattacharya for fruitful discussions. Special thanks to Dheeraj
Kulkarni for his valuable time, suggestions, proofreading and
Figure~\ref{fig:CO-UC-tops-same}. We thank our funding institutions:
the first author was supported by SERB's grant SRG/2020/001823 and the
second one by CSIR grant~09/1020(0159)/2019-EMR-I.



\end{document}